\documentclass[onefignum,onetabnum,
]{siamart251216}

\usepackage{bm}
\usepackage{amssymb}

\usepackage{tikz}

\usepackage{booktabs}
\usepackage{array}

\usepackage{algorithm}
\usepackage{algpseudocode}

\newsiamremark{question}{Question}

\headers{Everything is Vecchia}{E. Kaminetz \& R.J. Webber}

\title{Everything is Vecchia: \\ Unifying low-rank and sparse inverse Cholesky approximations
}

\author{Eagan Kaminetz\thanks{University of California San Diego, La Jolla, CA
  (\email{ekaminetz@ucsd.edu}, \email{rwebber@ucsd.edu}).}
\and Robert J. Webber\footnotemark[2]
}

\ifpdf
\hypersetup{
  pdftitle={Everything is Vecchia},
  pdfauthor={E. Kaminetz \& R.J. Webber}
}
\fi

\begin{document}

\maketitle

\begin{abstract}
The partial pivoted Cholesky approximation accurately represents matrices that are close to being low-rank.
Meanwhile, the Vecchia approximation accurately represents matrices with inverse Cholesky factors that are close to being sparse.
What happens if a partial Cholesky approximation is combined with a Vecchia approximation of the residual?
This paper shows how the sum is exactly a Vecchia approximation of the original matrix with an augmented sparsity pattern. 
Thus, Vecchia approximations subsume a class of existing matrix approximations and have broad applicability.
\end{abstract}



\section{Motivation}
\label{sec:intro}

The goal of this paper is to approximate a large, dense, positive-semidefinite matrix $\bm{A} \in \mathbb{C}^{n \times n}$ by looking up and processing individual entries $\bm{A}(i,j)$.
As the main example, $\bm{A}$ might be the kernel matrix for a high-dimensional machine learning data set \cite{scholkopf2001learning}. 
Since a kernel matrix can be very large (e.g., $n \geq 10^5$), kernel computations require an approximation $\hat{\bm{A}} \approx \bm{A}$ that is generated in \emph{linear} or \emph{sublinear} time.
A linear-time algorithm runs in $\mathcal{O}(n^2)$ arithmetic operations, comparable to the cost of looking at each entry of $\bm{A}$ once.
A sublinear-time algorithm runs in $o(n^2)$ operations, less than the cost of a single pass over $\bm{A}$'s entries.

Two approximations that can be generated in linear or sublinear time are the partial pivoted Cholesky \cite{chen2025randomly} and Vecchia \cite{vecchia1988estimation} methods.
They are traditionally regarded as accurate approximations for different types of matrices.
\begin{itemize}
\item Partial pivoted Cholesky provides an accurate approximation when the target matrix is close to being low-rank \cite[Thm.~2.3]{chen2025randomly}.
\item Vecchia provides an accurate approximation when the inverse Cholesky factor is close to being sparse \cite[Sec.~B.2]{schafer2021sparse}.
\end{itemize}
This paper unifies the two approaches.
It shows that the partial Cholesky approximation followed by a Vecchia approximation of the residual is exactly equivalent to a Vecchia approximation of the original matrix with an augmented sparsity pattern (\cref{thm:everything_vecchia}).
As a computational benefit, the hybrid method generates Vecchia approximations with $r$ nonzeros per row in $\mathcal{O}(rn)$ entry accesses rather than $\mathcal{O}(r^2 n)$, making them more practical for large kernel matrices.

The partial Cholesky + Vecchia approach, which appeared in past work \cite{zhao2024adaptive,cai2025posterior}, is both theoretically elegant and practically effective.
The next subsections present highlights of this paper's theory (\cref{sec:theory}) and experiments (\cref{sec:experiments}).

\subsection{Theoretical optimality} \label{sec:theory}

This paper asks and answers, ``In what way is the Vecchia approximation optimal?''
Optimality theory for the Vecchia approximation was developed in past work \cite{vecchia1988estimation,kaporin1994new,axelsson2000sublinear,yeremin2000factorized,schafer2021sparse}, but this paper presents a new extension to positive-semidefinite matrices and new error bounds for linear solves and determinant calculations.

The optimality analysis is based on the following Kaporin condition number \cite{kaporin1994new,axelsson2000sublinear}, which measures the accuracy of a matrix approximation.
Here we slightly extend the traditional definition to handle positive-semidefinite matrices.
\begin{definition}[Kaporin condition number] \label{def:kaporin}
    For any positive-semidefinite matrix $\bm{A} \in \mathbb{C}^{n \times n}$ and any positive-semidefinite approximation $\hat{\bm{A}} \in \mathbb{C}^{n \times n}$, the Kaporin condition number is
    \begin{equation*}
        \kappa_{\rm Kap} = \frac{\bigl(\frac{1}{r} \operatorname{tr} ( \bm{A} \hat{\bm{A}}^+)\bigr)^{r}}{\operatorname{vol} (\bm{A} \hat{\bm{A}}^+)},
        \qquad \text{where } r = \operatorname{rank}(\bm{A}),
    \end{equation*}
    if $\bm{A}$ and $\hat{\bm{A}}$ have the same range.
    The Kaporin condition number is $\kappa_{\rm Kap} = \infty$ if $\bm{A}$ and $\hat{\bm{A}}$ have different ranges.
    Here, the volume is the product of the positive eigenvalues.
\end{definition}
The Kaporin condition number is the arithmetic mean of the positive eigenvalues of $\bm{A} \hat{\bm{A}}^+$ raised to the $\operatorname{rank}(\bm{A})$ power, divided by the product of the eigenvalues.
Intuitively, $\kappa_{\rm Kap}$ measures how well the approximation preserves the spectrum of $\bm{A}$; $\kappa_{\rm Kap} = 1$ corresponds to exact recovery up to a constant multiple, while large values indicate distortion of the eigenvalues.

Kaporin \cite[App.~A.3]{kaporin1994new} established that the Vecchia approximation achieves the smallest possible Kaporin condition number for any strictly positive-definite target matrix and any sparsity pattern.
See \cref{thm:optimality} for an extension of this result to positive-semidefinite matrices.

The Kaporin optimality result is important because the Kaporin condition number controls the error of several linear algebra calculations.
\Cref{tab:results} shows that a smaller Kaporin condition number directly improves two common tasks: solving linear systems and estimating determinants.
Full descriptions of the linear algebra calculations and Kaporin condition number bounds are in \cref{sec:optimality}.

\begin{table}[t]
\centering
\caption{Summary of error bounds in terms of the Kaporin condition number $\kappa_{\rm Kap}$.
The direct solver bounds require the normalization $\operatorname{tr}\bigl(\bm{A} \hat{\bm{A}}^+\bigr) = \operatorname{rank}(\bm{A})$.
The determinant bounds require that $\bm{A}$ and $\hat{\bm{A}}$ are strictly positive-definite, and the iterative determinant bound requires $\log(\kappa_{\rm Kap}) \leq n$.} \label{tab:results}
\begin{tabular}{
    >{\raggedright\arraybackslash}m{0.18\linewidth}  
    >{\raggedright\arraybackslash}m{0.46\linewidth}  
    >{\raggedright\arraybackslash}m{0.24\linewidth}  
}
\toprule
\textbf{Method} & \textbf{Error bound} & \textbf{Reference} \\
\midrule
Linear system, direct solver &
$\displaystyle
\frac{\lVert \hat{\bm{x}} - \bm{x}_{\star} \rVert_{\bm{A}}^2}
     {\lVert \bm{x}_0 - \bm{x}_{\star} \rVert_{\bm{A}}^2}
 \leq 
 2 \operatorname{rank}(\bm{A})
 \log(\kappa_{\mathrm{Kap}})$
 & \cref{prop:direct}
\\[10pt]
Linear system, iterative solver &
$\displaystyle
\frac{\lVert \bm{x}_t - \bm{x}_{\star} \rVert_{\bm{A}}^2}
     {\lVert \bm{x}_0 - \bm{x}_{\star} \rVert_{\bm{A}}^2}
 \leq 
 \biggl[\frac{3\log(\kappa_{\mathrm{Kap}})}{t}\biggr]^t$
 & \cite{axelsson2000sublinear}, \cref{prop:pcg}
\\[10pt]
Determinant, direct solver &
$\displaystyle
\log\biggl(\frac{\det \hat{\bm{A}}}{\det \bm{A}}\biggr)
= \log(\kappa_{\mathrm{Kap}})$
 & \cref{prop:det}
\\[10pt]
Determinant, iterative solver &
$\displaystyle
\mathbb{E}\biggl|\log\biggl(
\frac{{\rm e}^{s_t}\det \hat{\bm{A}}}{\det \bm{A}}
\biggr)
\biggr|^2
\leq
\frac{4\log(\kappa_{\mathrm{Kap}})}{t}$
 & \cref{prop:stochastic_det}
\\[10pt]
\bottomrule
\end{tabular}
\end{table}

\subsection{Empirical performance}
\label{sec:experiments}

To test the performance of partial Cholesky + Vecchia, we downloaded $22$ machine learning data sets and subsampled each data set to $n = 20{,}000$ data points containing $d \in [4, 784]$ normalized predictors; see \cref{tab:datasets} for a list of data sets.
Using the data points $\bm{z}_1, \ldots, \bm{z}_n \in \mathbb{R}^d$, we defined the strictly positive-definite $n \times n$ kernel matrix with entries
\begin{equation*}
    \bm{A}(i,j) = \exp\biggl(-\frac{\lVert \bm{z}_i - \bm{z}_j \rVert^2}{2d}\biggr) + \mu\, \delta(i,j),
    \qquad \text{for } \mu \in \{10^{-3}, 10^{-6}, 10^{-10}\}.
\end{equation*}
Then we applied conjugate gradient to iteratively solve linear systems $\bm{A} \bm{x} = \bm{b}$ using various  matrix approximations $\hat{\bm{A}}$ as preconditioners.
We declared each problem ``solved'' as soon as the relative error reached a tolerance $\lVert \bm{A} \hat{\bm{x}} - \bm{b} \rVert / \lVert \bm{b} \rVert \leq 10^{-3}$.
Full details are in \cref{sec:empirical}.

\begin{figure}[t]
    \centering
    \includegraphics[width=\linewidth]{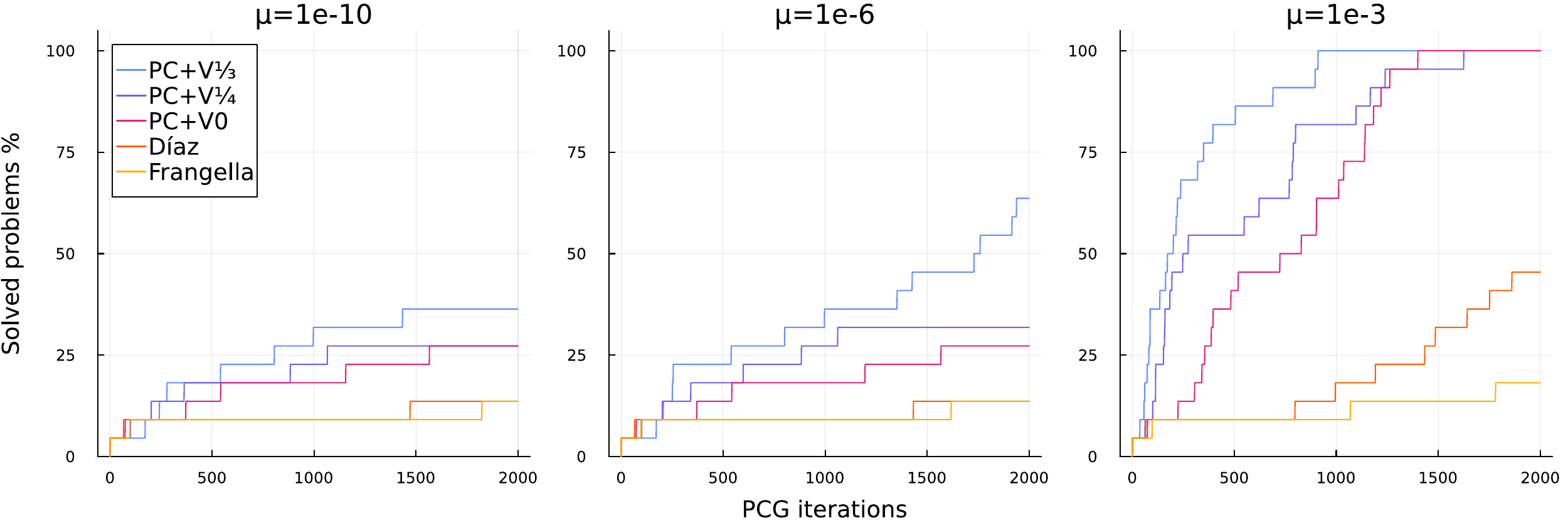}
    \caption{Comparison of five preconditioners
    based on randomly pivoted Cholesky \cite{chen2025randomly} with rank $r = \lfloor n^{1/2} \rfloor = 141$.
    Partial Cholesky + Vecchia preconditioners use either $q = 0$ (\textbf{PC+V0}), $q = \lfloor n^{1/4} \rfloor = 11$ (\textbf{PC+V1/4}), or $q = \lfloor n^{1/3} \rfloor = 27$ (\textbf{PC+V1/3}) nonzers per row in the Vecchia component.
    Alternative preconditioners (\textbf{D{\'i}az} \cite{diaz2024robust} and \textbf{Frangella} \cite{frangella2023randomized}) modify randomly pivoted Cholesky by adding a multiple of the identity to the approximation and/or its nullspace.}
    \label{fig:intro}
\end{figure}

\Cref{fig:intro} shows that partial Cholesky + Vecchia preconditioners consistently outperform existing Cholesky-based methods \cite{frangella2023randomized,diaz2024robust}, solving up to $11\times$ as many problems within $t = 1000$ iterations.
Among partial Cholesky + Vecchia preconditioners, raising the number of off-diagonal nonzeros per row in the Vecchia component from $q = 0$ to $q = \lfloor n^{1/3}\rfloor = 27$ increases the number of solved problems by $1.6$--$2.0\times$.
Despite these gains, forming effective preconditioners for near-singular matrices remains an open problem because \emph{no} available preconditioner can solve half the target problems given the smallest regularization parameter $\mu = 10^{-10}$.

In conclusion, partial Cholesky + Vecchia provides satisfactory approximations to many but not all positive-semidefinite kernel matrices that were previously inaccessible.
In the future, we are optimistic that we can continue to improve this approximation by creative efforts to optimize the sparsity pattern.

\subsection{Organization}

The rest of the paper is organized as follows.
\Cref{sec:notation} establishes notation.
\Cref{sec:background} proves that partial Cholesky + Vecchia = Vecchia.
\Cref{sec:optimality} presents optimality theory for the Vecchia approximation.
\Cref{sec:optimization} introduces optimization strategies for choosing the sparsity pattern in the Vecchia approximation,
and \cref{sec:empirical} presents numerical experiments.

\subsection{Notation} \label{sec:notation}

Scalars are written in lower case italics, e.g., $m,n,r \in \mathbb{N}$. Vectors are in lower case boldface, e.g., $\bm{u}, \bm{v} \in \mathbb{C}^n$.
Matrices are in boldface capital letters, e.g., $\bm{A}, \bm{B} \in \mathbb{C}^{n \times n}$.
Index sets are in sans serif font, e.g., $\mathsf{R}$, $\mathsf{S} \subseteq \{1, \ldots, n\}$.
We use $\bm{u}(i)$ to refer to to the $i$th entry of a vector $\bm{u}$,
and we use $\bm{A}(i,j)$ to refer to the $(i,j)$ entry of a matrix $\bm{A}$.
Given index sets $\mathsf{R},\mathsf{S}$, we use $\bm{u}(\mathsf{R})$ to refer to the subvector $(\bm{u}(i))_{i \in \mathsf{R}}$, and we use $\bm{A}(\mathsf{R},\mathsf{S})$ to refer to the submatrix $(\bm{A}(i,j))_{i \in \mathsf{R},\, j \in \mathsf{S}}$.
Additionally, $\bm{A}(i, \cdot)$ and $\bm{A}(\cdot, i)$ indicate the $i$th row and column of $\bm{A}$.

Following the standard conventions, $\bm{0}$ denotes a column vector, row vector, or matrix of all zeros; $\mathbf{I}$ is an identity matrix; and
$\bm{e}_i$ denotes a standard basis vector that which is all zeros except for a $1$ in the $i$th entry.
The complex conjugate of a scalar $u$ is $\overline{u}$. 
The conjugate transpose of a vector $\bm{v}$ or matrix $\bm{A}$ is $\bm{v}^*$ or $\bm{A}^*$. The inverse of $\bm{A}$ is $\bm{A}^{-1}$, the conjugate transpose inverse is $\bm{A}^{-*}$, and the Moore-Penrose pseudoinverse is $\bm{A}^+$.

Given any positive-semidefinite matrix $\bm{A} \in \mathbb{C}^{n \times n}$, the $\bm{A}$-weighted distance between two vectors $\bm{u}$ and $\bm{v}$ is
\begin{equation*}
    d_{\bm{A}}(\bm{u}, \bm{v}) = \bigl[(\bm{u} - \bm{v})^* \bm{A} (\bm{u} - \bm{v})\bigr]^{1/2}.
\end{equation*}
The $\bm{A}$-weighted distance between a vector $\bm{u}$ and a set of vectors $\mathsf{V}$ is
\begin{equation*}
    d_{\bm{A}}(\bm{u}, \mathsf{V}) = \inf_{\bm{v} \in \mathsf{V}} d_{\bm{A}}(\bm{u}, \bm{v}).
\end{equation*}
Last, the volume is the product of the positive eigenvalues of $\bm{A}$,
\begin{equation*}
    \operatorname{vol}(\bm{A}) = \prod_{\lambda_i(\bm{A}) > 0} \lambda_i(\bm{A}).
\end{equation*}
The volume is the same as the determinant if $\bm{A}$ is strictly positive-definite.

\section{Partial Cholesky + Vecchia = Vecchia} \label{sec:background}

This section introduces the framework of sparse Cholesky approximation (\cref{sec:cholesky_background}).
Then it proves that partial Cholesky + Vecchia = Vecchia (\cref{sec:main_result}) and discusses broader implications (\cref{sec:implications}).

\subsection{Factored approximations based on the Cholesky decomposition} \label{sec:cholesky_background}

Any positive-semidefinite matrix can be exactly represented through a Cholesky or inverse Cholesky decomposition.
See the below definitions and see \cref{fig:cholesky_ill} for illustrations.
\begin{definition}[Pivoted Cholesky and inverse Cholesky decompositions]
For any positive-semidefinite matrix $\bm{A} \in \mathbb{C}^{n \times n}$ and permutation matrix $\bm{P} \in \{0, 1\}^{n \times n}$, the pivoted Cholesky and pivoted inverse Cholesky decompositions are defined as
\begin{equation}
\label{eq:exact_cholesky}
	\bm{A} = \bm{P} \bm{L} \bm{D} \bm{L}^* \bm{P}^*
	\quad \text{and} \quad
	\bm{A} = \bm{P} \bm{C}^{-1} \bm{D} \bm{C}^{-*} \bm{P}^*.
\end{equation}
Here, $\bm{D} \in \mathbb{R}_+^{n \times n}$ is a nonnegative-valued diagonal matrix while $\bm{L} \in \mathbb{C}^{n \times n}$ and $\bm{C} \in \mathbb{C}^{n \times n}$ are lower triangular matrices with ones on the diagonal.
\end{definition}
The Cholesky decomposition can be manipulated into inverse Cholesky form and vice versa through a matrix inversion $\bm{C} = \bm{L}^{-1}$.

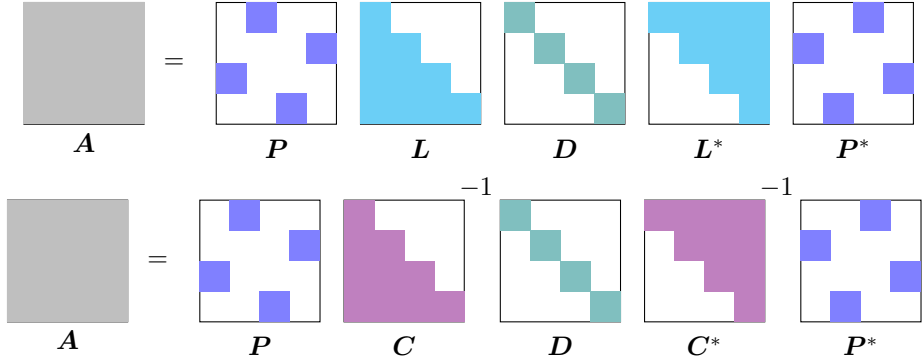
\begin{figure}[t]
\centering
\begin{tikzpicture}[scale=0.8]

\tikzset{
  matrix box/.style={draw=black, thin},
  nz/.style={fill opacity=1.0},
}

\begin{scope}[shift={(0,0)}]
  \draw[matrix box] (0,0) rectangle (2,2);
  \foreach \i in {0,1,2,3} {
    \foreach \j in {0,1,2,3} {
      \fill[nz, fill=gray!50] (0.5*\j,0.5*\i) rectangle (0.5*\j+0.5,0.5*\i+0.5);
    }
  }
  \node at (1,-0.3) {$\bm{A}$};
\end{scope}

\node at (2.5,1) {$=$};

\begin{scope}[shift={(3.2,0)}]
  \draw[matrix box] (0,0) rectangle (2,2);
  \fill[nz, fill=blue!50] (0.5,1.5) rectangle (1.0,2.0); 
  \fill[nz, fill=blue!50] (1.5,1.0) rectangle (2.0,1.5); 
  \fill[nz, fill=blue!50] (0.0,0.5) rectangle (0.5,1.0); 
  \fill[nz, fill=blue!50] (1.0,0.0) rectangle (1.5,0.5); 
  \node at (1,-0.4) {$\bm{P}$};
\end{scope}

\begin{scope}[shift={(5.6,0)}]
  \draw[matrix box] (0,0) rectangle (2,2);
  \fill[nz, fill=cyan!50] (0.0,1.5) rectangle (0.5,2.0); 
  \fill[nz, fill=cyan!50] (0.0,1.0) rectangle (0.5,1.5); 
  \fill[nz, fill=cyan!50] (0.5,1.0) rectangle (1.0,1.5); 
  \fill[nz, fill=cyan!50] (0.0,0.5) rectangle (0.5,1.0); 
  \fill[nz, fill=cyan!50] (0.5,0.5) rectangle (1.0,1.0); 
  \fill[nz, fill=cyan!50] (1.0,0.5) rectangle (1.5,1.0); 
  \fill[nz, fill=cyan!50] (0.0,0.0) rectangle (0.5,0.5); 
  \fill[nz, fill=cyan!50] (0.5,0.0) rectangle (1.0,0.5); 
  \fill[nz, fill=cyan!50] (1.0,0.0) rectangle (1.5,0.5); 
  \fill[nz, fill=cyan!50] (1.5,0.0) rectangle (2.0,0.5); 
  \node at (1,-0.4) {$\bm{L}$};
\end{scope}

\begin{scope}[shift={(8.0,0)}]
  \draw[matrix box] (0,0) rectangle (2,2);
  \fill[nz, fill=teal!50] (0.0,1.5) rectangle (0.5,2.0); 
  \fill[nz, fill=teal!50] (0.5,1.0) rectangle (1.0,1.5); 
  \fill[nz, fill=teal!50] (1.0,0.5) rectangle (1.5,1.0); 
  \fill[nz, fill=teal!50] (1.5,0.0) rectangle (2.0,0.5); 
  \node at (1,-0.4) {$\bm{D}$};
\end{scope}

\begin{scope}[shift={(10.4,0)}]
  \draw[matrix box] (0,0) rectangle (2,2);
  \fill[nz, fill=cyan!50] (0.0,1.5) rectangle (0.5,2.0); 
  \fill[nz, fill=cyan!50] (0.5,1.5) rectangle (1.0,2.0); 
  \fill[nz, fill=cyan!50] (1.0,1.5) rectangle (1.5,2.0); 
  \fill[nz, fill=cyan!50] (1.5,1.5) rectangle (2.0,2.0); 
  \fill[nz, fill=cyan!50] (0.5,1.0) rectangle (1.0,1.5); 
  \fill[nz, fill=cyan!50] (1.0,1.0) rectangle (1.5,1.5); 
  \fill[nz, fill=cyan!50] (1.5,1.0) rectangle (2.0,1.5); 
  \fill[nz, fill=cyan!50] (1.0,0.5) rectangle (1.5,1.0); 
  \fill[nz, fill=cyan!50] (1.5,0.5) rectangle (2.0,1.0); 
  \fill[nz, fill=cyan!50] (1.5,0.0) rectangle (2.0,0.5); 
  \node at (1,-0.4) {$\bm{L}^*$};
\end{scope}

\begin{scope}[shift={(12.8,0)}]
  \draw[matrix box] (0,0) rectangle (2,2);
  \fill[nz, fill=blue!50] (1.0,1.5) rectangle (1.5,2.0); 
  \fill[nz, fill=blue!50] (0.0,1.0) rectangle (0.5,1.5); 
  \fill[nz, fill=blue!50] (1.5,0.5) rectangle (2.0,1.0); 
  \fill[nz, fill=blue!50] (0.5,0.0) rectangle (1.0,0.5); 
  \node at (1,-0.4) {$\bm{P}^*$};
\end{scope}

\end{tikzpicture}
\begin{tikzpicture}[scale=0.8]

\tikzset{
  matrix box/.style={draw=black, thin},
  nz/.style={fill opacity=1.0},
}

\begin{scope}[shift={(0,0)}]
  \draw[matrix box] (0,0) rectangle (2,2);
  \foreach \i in {0,1,2,3} {
    \foreach \j in {0,1,2,3} {
      \fill[nz, fill=gray!50] (0.5*\j,0.5*\i) rectangle (0.5*\j+0.5,0.5*\i+0.5);
    }
  }
  \node at (1,-0.3) {$\bm{A}$};
\end{scope}

\node at (2.5,1) {$=$};

\begin{scope}[shift={(3.2,0)}]
  \draw[matrix box] (0,0) rectangle (2,2);
  \fill[nz, fill=blue!50] (0.5,1.5) rectangle (1.0,2.0); 
  \fill[nz, fill=blue!50] (1.5,1.0) rectangle (2.0,1.5); 
  \fill[nz, fill=blue!50] (0.0,0.5) rectangle (0.5,1.0); 
  \fill[nz, fill=blue!50] (1.0,0.0) rectangle (1.5,0.5); 
  \node at (1,-0.4) {$\bm{P}$};
\end{scope}

\begin{scope}[shift={(5.6,0)}]
  \draw[matrix box] (0,0) rectangle (2,2);
  \fill[nz, fill=violet!50] (0.0,1.5) rectangle (0.5,2.0); 
  \fill[nz, fill=violet!50] (0.0,1.0) rectangle (0.5,1.5); 
  \fill[nz, fill=violet!50] (0.5,1.0) rectangle (1.0,1.5); 
  \fill[nz, fill=violet!50] (0.0,0.5) rectangle (0.5,1.0); 
  \fill[nz, fill=violet!50] (0.5,0.5) rectangle (1.0,1.0); 
  \fill[nz, fill=violet!50] (1.0,0.5) rectangle (1.5,1.0); 
  \fill[nz, fill=violet!50] (0.0,0.0) rectangle (0.5,0.5); 
  \fill[nz, fill=violet!50] (0.5,0.0) rectangle (1.0,0.5); 
  \fill[nz, fill=violet!50] (1.0,0.0) rectangle (1.5,0.5); 
  \fill[nz, fill=violet!50] (1.5,0.0) rectangle (2.0,0.5); 
  \node at (1,-0.4) {$\bm{C}$};
  \node at (2.2,2.2) {${-1}$};
\end{scope}

\begin{scope}[shift={(8.2,0)}]
  \draw[matrix box] (0,0) rectangle (2,2);
  \fill[nz, fill=teal!50] (0.0,1.5) rectangle (0.5,2.0); 
  \fill[nz, fill=teal!50] (0.5,1.0) rectangle (1.0,1.5); 
  \fill[nz, fill=teal!50] (1.0,0.5) rectangle (1.5,1.0); 
  \fill[nz, fill=teal!50] (1.5,0.0) rectangle (2.0,0.5); 
  \node at (1,-0.4) {$\bm{D}$};
\end{scope}

\begin{scope}[shift={(10.6,0)}]
  \draw[matrix box] (0,0) rectangle (2,2);
  \fill[nz, fill=violet!50] (0.0,1.5) rectangle (0.5,2.0); 
  \fill[nz, fill=violet!50] (0.5,1.5) rectangle (1.0,2.0); 
  \fill[nz, fill=violet!50] (1.0,1.5) rectangle (1.5,2.0); 
  \fill[nz, fill=violet!50] (1.5,1.5) rectangle (2.0,2.0); 
  \fill[nz, fill=violet!50] (0.5,1.0) rectangle (1.0,1.5); 
  \fill[nz, fill=violet!50] (1.0,1.0) rectangle (1.5,1.5); 
  \fill[nz, fill=violet!50] (1.5,1.0) rectangle (2.0,1.5); 
  \fill[nz, fill=violet!50] (1.0,0.5) rectangle (1.5,1.0); 
  \fill[nz, fill=violet!50] (1.5,0.5) rectangle (2.0,1.0); 
  \fill[nz, fill=violet!50] (1.5,0.0) rectangle (2.0,0.5); 
  \node at (1,-0.4) {$\bm{C}^*$};
  \node at (2.2,2.2) {${-1}$};
\end{scope}

\begin{scope}[shift={(13.2,0)}]
  \draw[matrix box] (0,0) rectangle (2,2);
  \fill[nz, fill=blue!50] (1.0,1.5) rectangle (1.5,2.0); 
  \fill[nz, fill=blue!50] (0.0,1.0) rectangle (0.5,1.5); 
  \fill[nz, fill=blue!50] (1.5,0.5) rectangle (2.0,1.0); 
  \fill[nz, fill=blue!50] (0.5,0.0) rectangle (1.0,0.5); 
  \node at (1,-0.4) {$\bm{P}^*$};
\end{scope}

\end{tikzpicture}
\caption{Cholesky and inverse Cholesky decompositions of a dense matrix $\bm{A}$.
Filled boxes show entries that are allowed to be nonzero.
}
\label{fig:cholesky_ill}
\end{figure}

The Cholesky and inverse Cholesky decompositions facilitate fast matrix computations.
Once the factors $\bm{P}, \bm{D}, \bm{L}$ or $\bm{P}, \bm{D}, \bm{C}$ have been generated and stored, the matrix $\bm{A}$ does not need to be accessed again.
Any matrix--vector product $\bm{v} \mapsto \bm{A} \bm{v}$ can be computed by sequentially multiplying the vector with each matrix in the factorization \cref{eq:exact_cholesky}, and the multiplications with $\bm{C}^{-1}$ or $\bm{C}^{-*}$ can be carried out via efficient triangular solves, without computing the inverse explicitly.
Additionally, any consistent linear system $\bm{A} \bm{x} = \bm{b}$ can be solved by multiplying the output vector $\bm{b}$ by a sequence of matrices,
\begin{equation*}
    \bm{x} = \bm{P} \bm{L}^{-*} \bm{D}^+ \bm{L}^{-1} \bm{P}^* \bm{b}
    \quad \text{or} \quad
    \bm{x} = \bm{P} \bm{C}^* \bm{D}^+ \bm{C} \bm{P}^* \bm{b}.
\end{equation*}
Therefore, the cost of a matrix-vector product or linear solve is $\mathcal{O}(n^2)$ arithmetic operations.

Motivated by the Cholesky and inverse Cholesky decompositions, we can generate a sparse Cholesky or sparse inverse Cholesky approximation
\begin{equation}
\label{eq:inexact_cholesky}
	\hat{\bm{A}} = \bm{P} \hat{\bm{L}} \hat{\bm{D}} \hat{\bm{L}}^* \bm{P}^*
	\quad \text{or} \quad
	\hat{\bm{A}} = \bm{P} \hat{\bm{C}}^{-1} \hat{\bm{D}} \hat{\bm{C}}^{-*} \bm{P}^*.
\end{equation}
Here $\bm{P}$ is a permutation matrix, $\hat{\bm{D}}$ is a nonnegative-valued diagonal matrix, and $\hat{\bm{L}}$ or $\hat{\bm{C}}$ is a sparse lower triangular matrix with ones on the diagonal.
The \emph{sparsity pattern} $(\mathsf{S}_i)_{i=1}^n$ is a sequence of \emph{sparsity index sets} $\mathsf{S}_i \subseteq \{1, \ldots, i-1\}$ describing which off-diagonal entries in the rows of $\hat{\bm{L}}$ or $\hat{\bm{C}}$ are allowed to be nonzero.

The imposition of sparsity leads to three benefits.
First, the sparse approximation factors can be stored in $\mathcal{O}(s n)$ memory.
Second, matrix--vector products and linear solves can be computed in $\mathcal{O}(sn)$ arithmetic operations, where $s$ is an upper bound on the cardinality of the sparsity pattern: $|\mathsf{S}_i| \leq s$ for each $i = 1, \ldots, n$.
Third, in many cases generating the approximation $\hat{\bm{A}}$ is relatively cheap, as it only requires examining $\mathcal{O}(s n)$ or $\mathcal{O}(s^2 n)$ entries of $\bm{A}$ and performing $\mathcal{O}(s^2 n)$ or $\mathcal{O}(s^3 n)$ additional arithmetic operations.

\Cref{sec:partial,sec:vecchia} describe two types of sparse Cholesky approximations that can be generated in sublinear or linear time.

\subsubsection{Partial pivoted Cholesky} \label{sec:partial}

The partial pivoted Cholesky approximation is a common rank-revealing factorization for positive-semidefinite matrices \cite{golub1965numerical}.
See the following definition, and see \cref{fig:partial} for an illustration.

\begin{figure}[t]
\centering
\begin{tikzpicture}[scale=0.8]

\tikzset{
  matrix box/.style={draw=black, thin},
  nz/.style={fill opacity=1.0},
}

\begin{scope}[shift={(0,0)}]
  \draw[matrix box] (0,0) rectangle (2,2);
  \fill[nz, fill=gray!50] (0.0,1.5) rectangle (0.5,2.0); 
  \fill[nz, fill=gray!50] (0.5,1.5) rectangle (1.0,2.0); 
  \fill[nz, fill=gray!50] (1.0,1.5) rectangle (1.5,2.0); 
  \fill[nz, fill=gray!50] (1.5,1.5) rectangle (2.0,2.0); 
  \fill[nz, fill=gray!50] (0.0,1.0) rectangle (0.5,1.5); 
  \fill[nz, fill=gray!50] (1.0,1.0) rectangle (1.5,1.5); 
  \fill[nz, fill=gray!50] (0.0,0.5) rectangle (0.5,1.0); 
  \fill[nz, fill=gray!50] (0.5,0.5) rectangle (1.0,1.0); 
  \fill[nz, fill=gray!50] (1.0,0.5) rectangle (1.5,1.0); 
  \fill[nz, fill=gray!50] (1.5,0.5) rectangle (2.0,1.0); 
  \fill[nz, fill=gray!50] (0.0,0.0) rectangle (0.5,0.5); 
  \fill[nz, fill=gray!50] (1.0,0.0) rectangle (1.5,0.5); 
  \node at (1,-0.4) {$\bm{A}$};
\end{scope}

\node at (2.5,1) {$\rightarrow$};

\begin{scope}[shift={(3.2,0)}]
  \draw[matrix box] (0,0) rectangle (2,2);
  \fill[nz, fill=blue!50] (0.5,1.5) rectangle (1.0,2.0); 
  \fill[nz, fill=blue!50] (1.5,1.0) rectangle (2.0,1.5); 
  \fill[nz, fill=blue!50] (0.0,0.5) rectangle (0.5,1.0); 
  \fill[nz, fill=blue!50] (1.0,0.0) rectangle (1.5,0.5); 
  \node at (1,-0.4) {$\bm{P}$};
\end{scope}

\begin{scope}[shift={(5.6,0)}]
  \draw[matrix box] (0,0) rectangle (2,2);
  \fill[nz, fill=cyan!50] (0.0,1.5) rectangle (0.5,2.0); 
  \fill[nz, fill=cyan!50] (0.0,1.0) rectangle (0.5,1.5); 
  \fill[nz, fill=cyan!50] (0.5,1.0) rectangle (1.0,1.5); 
  \fill[nz, fill=cyan!50] (0.0,0.5) rectangle (0.5,1.0); 
  \fill[nz, fill=cyan!50] (0.5,0.5) rectangle (1.0,1.0); 
  \fill[nz, fill=cyan!50] (1.0,0.5) rectangle (1.5,1.0); 
  \fill[nz, fill=cyan!50] (0.0,0.0) rectangle (0.5,0.5); 
  \fill[nz, fill=cyan!50] (0.5,0.0) rectangle (1.0,0.5); 
  \fill[nz, fill=cyan!50] (1.5,0.0) rectangle (2.0,0.5); 
  \node at (1,-0.4) {$\hat{\bm{L}}$};
\end{scope}

\begin{scope}[shift={(8.0,0)}]
  \draw[matrix box] (0,0) rectangle (2,2);
  \fill[nz, fill=teal!50] (0.0,1.5) rectangle (0.5,2.0); 
  \fill[nz, fill=teal!50] (0.5,1.0) rectangle (1.0,1.5); 
  \node at (1,-0.4) {$\hat{\bm{D}}$};
\end{scope}

\begin{scope}[shift={(10.4,0)}]
  \draw[matrix box] (0,0) rectangle (2,2);
  \fill[nz, fill=cyan!50] (0.0,1.5) rectangle (0.5,2.0); 
  \fill[nz, fill=cyan!50] (1.0,1.5) rectangle (1.5,2.0); 
  \fill[nz, fill=cyan!50] (1.5,1.5) rectangle (2.0,2.0); 
  \fill[nz, fill=cyan!50] (0.5,1.0) rectangle (1.0,1.5); 
  \fill[nz, fill=cyan!50] (1.0,1.0) rectangle (1.5,1.5); 
  \fill[nz, fill=cyan!50] (1.5,1.0) rectangle (2.0,1.5); 
  \fill[nz, fill=cyan!50] (1.0,0.5) rectangle (1.5,1.0); 
  \fill[nz, fill=cyan!50] (1.5,0.5) rectangle (2.0,1.0); 
  \fill[nz, fill=cyan!50] (1.5,0.0) rectangle (2.0,0.5); 
  \node at (1,-0.4) {$\hat{\bm{L}}^*$};
\end{scope}

\begin{scope}[shift={(12.8,0)}]
  \draw[matrix box] (0,0) rectangle (2,2);
  \fill[nz, fill=blue!50] (1.0,1.5) rectangle (1.5,2.0); 
  \fill[nz, fill=blue!50] (0.0,1.0) rectangle (0.5,1.5); 
  \fill[nz, fill=blue!50] (1.5,0.5) rectangle (2.0,1.0); 
  \fill[nz, fill=blue!50] (0.5,0.0) rectangle (1.0,0.5); 
  \node at (1,-0.4) {$\bm{P}^*$};
\end{scope}

\end{tikzpicture}
\caption{Partial pivoted Cholesky accesses the gray-colored entries of $\bm{A}$.
Here, the approximation rank is $r = 2$ and the columns $u_1 = 3$ and $u_2 = 1$ are perfectly replicated.}
\label{fig:partial}
\end{figure}
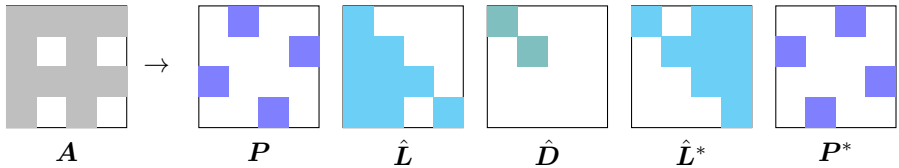

\begin{definition}[Partial pivoted Cholesky]
\label{def:cholesky_approximation}
    Given a positive-semidefinite matrix $\bm{A} \in \mathbb{C}^{n \times n}$, the partial pivoted Cholesky approximation with permutation $\bm{P}$ and approximation rank $r$ is a sparse Cholesky approximation
    $\hat{\bm{A}} = \bm{P} \hat{\bm{L}} \hat{\bm{D}} \hat{\bm{L}}^* \bm{P}^*$ where each row $\hat{\bm{L}}(i, \cdot)$ has sparsity pattern 
    \begin{equation*}
    \mathsf{S}_i = \{1, \ldots, r\} \cap \{1, \ldots i-1\}.
    \end{equation*}
    In this approximation, the first $r$ rows and columns of $\hat{\bm{L}} \hat{\bm{D}} \hat{\bm{L}}^*$ match the first $r$ rows and columns of the permuted matrix $\tilde{\bm{A}} = \bm{P}^* \bm{A} \bm{P}$, and the last $n - r$ diagonal entries of $\hat{\bm{D}}$ are all zeros.
\end{definition}

\begin{algorithm}[t]
    \caption{Partial pivoted Cholesky approximation} \label{alg:partial}
    \begin{algorithmic}
	\Require{Positive-semidefinite matrix $\bm{A} \in \mathbb{C}^{n \times n}$ with entry-wise access; permutation matrix $\bm{P} \in \{0, 1\}^{n \times n}$; approximation rank $r$}
	\Ensure{Sparse Cholesky approximation $\hat{\bm{A}} = \bm{P} \hat{\bm{L}} \hat{\bm{D}} \hat{\bm{L}}^* \bm{P}^*$ in factored form}
	\State{Initialize sparse matrices $\hat{\bm{F}} = \bm{P}$ and $\hat{\bm{D}} = \bm{0} \in \mathbb{C}^{n \times n}$}
	\Comment{$\hat{\bm{F}} = \bm{P} \hat{\bm{L}}$}
	\For{$i = 1, \ldots, r$}
	\State{Identify pivot $u_i \in \{1, \ldots, n\}$ with $\bm{P}(u_i, i) = 1$}
	\State{$\bm{v} \leftarrow \bm{A}(\cdot, u_i) - \hat{\bm{F}} \hat{\bm{D}} \hat{\bm{F}}(u_i, \cdot)^*$}
    \If{$\bm{v}(u_i) > 0$}
	\State{$\hat{\bm{F}}(\cdot, i) \leftarrow \bm{v} / \bm{v}(u_i)$}
	\State{$\hat{\bm{D}}(i, i) \leftarrow \bm{v}(u_i)$}
    \EndIf
	\EndFor
	\State{$\hat{\bm{L}} = \bm{P}^* \hat{\bm{F}}$}
    \end{algorithmic}
\end{algorithm}

\Cref{alg:partial} generates a partial Cholesky approximation using $\mathcal{O}(rn)$ entry lookups and $\mathcal{O}(r^2 n)$ additional arithmetic operations.
The algorithm extracts a sequence of \emph{pivots} $u_1, \ldots, u_r$, which are indicated by the nonzero entries in the permutation matrix: $\bm{P}(u_i, i) = 1$ for each $i = 1, \ldots, r$.
Then it generates a rank-$r$ positive-semidefinite matrix that exactly replicates the chosen columns $u_1, \ldots, u_r$.
The method requires $\mathcal{O}(r^2 n)$ operations because each step $i$ requires forming linear combinations of the first $i$ selected columns.

A long line of research has investigated pivot selection strategies for the partial Cholesky approximation.
For example, ``randomly pivoted Cholesky'' is a randomized selection rule that guarantees near-optimal approximation error in the expected trace norm \cite{chen2025randomly}, and column-pivoted Cholesky is an alternative pivot selection rule based on greedy selection \cite{golub1965numerical}. See \cref{sec:partial_plus} for more analysis and discussion.

\subsubsection{Vecchia approximation} \label{sec:vecchia}

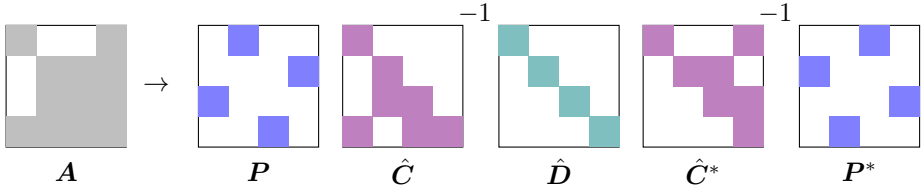
\begin{figure}[t]
\centering
\begin{tikzpicture}[scale=0.8]

\tikzset{
  matrix box/.style={draw=black, thin},
  nz/.style={fill opacity=1.0},
}

\begin{scope}[shift={(0,0)}]
  \draw[matrix box] (0,0) rectangle (2,2);
  \fill[nz, fill=gray!50] (0.0,1.5) rectangle (0.5,2.0); 
  \fill[nz, fill=gray!50] (1.5,1.5) rectangle (2.0,2.0); 
  \fill[nz, fill=gray!50] (0.5,1.0) rectangle (1.0,1.5); 
  \fill[nz, fill=gray!50] (1.0,1.0) rectangle (1.5,1.5); 
  \fill[nz, fill=gray!50] (1.5,1.0) rectangle (2.0,1.5); 
  \fill[nz, fill=gray!50] (0.5,0.5) rectangle (1.0,1.0); 
  \fill[nz, fill=gray!50] (1.0,0.5) rectangle (1.5,1.0); 
  \fill[nz, fill=gray!50] (1.5,0.5) rectangle (2.0,1.0); 
  \fill[nz, fill=gray!50] (0.0,0.0) rectangle (0.5,0.5); 
  \fill[nz, fill=gray!50] (0.5,0.0) rectangle (1.0,0.5); 
  \fill[nz, fill=gray!50] (1.0,0.0) rectangle (1.5,0.5); 
  \fill[nz, fill=gray!50] (1.5,0.0) rectangle (2.0,0.5); 
  \node at (1,-0.4) {$\bm{A}$};
\end{scope}

\node at (2.5,1) {$\rightarrow$};

\begin{scope}[shift={(3.2,0)}]
  \draw[matrix box] (0,0) rectangle (2,2);
  \fill[nz, fill=blue!50] (0.5,1.5) rectangle (1.0,2.0); 
  \fill[nz, fill=blue!50] (1.5,1.0) rectangle (2.0,1.5); 
  \fill[nz, fill=blue!50] (0.0,0.5) rectangle (0.5,1.0); 
  \fill[nz, fill=blue!50] (1.0,0.0) rectangle (1.5,0.5); 
  \node at (1,-0.4) {$\bm{P}$};
\end{scope}

\begin{scope}[shift={(5.6,0)}]
  \draw[matrix box] (0,0) rectangle (2,2);
  \fill[nz, fill=violet!50] (0.0,1.5) rectangle (0.5,2.0); 
  \fill[nz, fill=violet!50] (0.5,1.0) rectangle (1.0,1.5); 
  \fill[nz, fill=violet!50] (0.5,0.5) rectangle (1.0,1.0); 
  \fill[nz, fill=violet!50] (1.0,0.5) rectangle (1.5,1.0); 
  \fill[nz, fill=violet!50] (0.0,0.0) rectangle (0.5,0.5); 
  \fill[nz, fill=violet!50] (1.0,0.0) rectangle (1.5,0.5); 
  \fill[nz, fill=violet!50] (1.5,0.0) rectangle (2.0,0.5); 
  \node at (1,-0.4) {$\hat{\bm{C}}$};
  \node at (2.2,2.2) {${-1}$};
\end{scope}

\begin{scope}[shift={(8.2,0)}]
  \draw[matrix box] (0,0) rectangle (2,2);
  \fill[nz, fill=teal!50] (0.0,1.5) rectangle (0.5,2.0); 
  \fill[nz, fill=teal!50] (0.5,1.0) rectangle (1.0,1.5); 
  \fill[nz, fill=teal!50] (1.0,0.5) rectangle (1.5,1.0); 
  \fill[nz, fill=teal!50] (1.5,0.0) rectangle (2.0,0.5); 
  \node at (1,-0.4) {$\hat{\bm{D}}$};
\end{scope}

\begin{scope}[shift={(10.6,0)}]
  \draw[matrix box] (0,0) rectangle (2,2);
  \fill[nz, fill=violet!50] (0.0,1.5) rectangle (0.5,2.0); 
  \fill[nz, fill=violet!50] (1.5,1.5) rectangle (2.0,2.0); 
  \fill[nz, fill=violet!50] (0.5,1.0) rectangle (1.0,1.5); 
  \fill[nz, fill=violet!50] (1.0,1.0) rectangle (1.5,1.5); 
  \fill[nz, fill=violet!50] (1.0,0.5) rectangle (1.5,1.0); 
  \fill[nz, fill=violet!50] (1.5,0.5) rectangle (2.0,1.0); 
  \fill[nz, fill=violet!50] (1.5,0.0) rectangle (2.0,0.5); 
  \node at (1,-0.4) {$\hat{\bm{C}}^{*}$};
  \node at (2.2,2.2) {${-1}$};
\end{scope}

\begin{scope}[shift={(13.2,0)}]
  \draw[matrix box] (0,0) rectangle (2,2);
  \fill[nz, fill=blue!50] (1.0,1.5) rectangle (1.5,2.0); 
  \fill[nz, fill=blue!50] (0.0,1.0) rectangle (0.5,1.5); 
  \fill[nz, fill=blue!50] (1.5,0.5) rectangle (2.0,1.0); 
  \fill[nz, fill=blue!50] (0.5,0.0) rectangle (1.0,0.5); 
  \node at (1,-0.4) {$\bm{P}^*$};
\end{scope}

\end{tikzpicture}
\caption{Vecchia approximation accesses the gray-colored entries in $\bm{A}$.
The sparsity pattern is $\mathsf{S}_1 = \emptyset$, $\mathsf{S}_2 = \emptyset$, $\mathsf{S}_3 = \{2\}$, $\mathsf{S}_4 = \{1, 3\}$.}
\label{fig:vecchia}
\end{figure}

The Vecchia approximation, named after the statistician Aldo Vecchia \cite{vecchia1988estimation}, is a common approximation strictly positive-definite covariance matrices in the Gaussian process literature.
The Vecchia approximation does not necessarily replicate any entries of $\bm{A}$, but it guarantees optimal approximation accuracy in terms of the Kaporin condition number (\cref{thm:optimality}).
See below for a mathematical definition, which extends to all positive-semidefinite matrices, and see \cref{fig:vecchia} for an illustration.

\begin{definition}[Vecchia approximation]
\label{def:vecchia_approximation}
Given a positive-semidefinite matrix $\bm{A} \in \mathbb{C}^{n \times n}$, 
the Vecchia approximation with permutation $\bm{P}$ and sparsity pattern $(\mathsf{S}_i)_{i=1}^n$ is a sparse inverse Cholesky approximation $\hat{\bm{A}} = \bm{P} \hat{\bm{C}}^{-1} \hat{\bm{D}} \hat{\bm{C}}^{-*} \bm{P}^*$ where each row $\hat{\bm{C}}(i, \cdot)$ has sparsity pattern $\mathsf{S}_i$, and it satisfies
\begin{equation}
\label{eq:formula}
\begin{bmatrix} \hat{\bm{C}}(i, \mathsf{S}_i) & 1 \end{bmatrix}
\begin{bmatrix}
    \tilde{\bm{A}}(\mathsf{S}_i, \mathsf{S}_i) & \tilde{\bm{A}}(\mathsf{S}_i, i) \\
    \tilde{\bm{A}}(i, \mathsf{S}_i) & \tilde{\bm{A}}(i, i)
\end{bmatrix}
=
\begin{bmatrix}
    \bm{0} & \hat{\bm{D}}(i, i)
\end{bmatrix}
\quad \text{for } \tilde{\bm{A}} = \bm{P}^* \bm{A} \bm{P}.
\end{equation}
\end{definition}

\begin{algorithm}[t]
\caption{Conventional Vecchia algorithm} \label{alg:vecchia}
\begin{algorithmic}
\Require{Positive-semidefinite matrix $\bm{A} \in \mathbb{C}^{n \times n}$ with entry-wise access; permutation matrix $\bm{P} \in \{0, 1\}^{n \times n}$; sparsity pattern $(\mathsf{S}_i)_{i=1}^n$}
\Ensure{Vecchia approximation $\hat{\bm{A}} = \bm{P} \hat{\bm{C}}^{-1} \hat{\bm{D}} \hat{\bm{C}}^{-*} \bm{P}^*$ in factored form}
\State{Initialize sparse matrices $\hat{\bm{C}} = \mathbf{I} \in \mathbb{C}^{n \times n}$ and $\hat{\bm{D}} = \bm{0} \in \mathbb{C}^{n \times n}$}
\For{$i = 1, \ldots, n$}
\Comment{Can be executed in parallel}
\State{$\begin{bmatrix}
\bm{M} & \bm{v} \\
\bm{v}^* & \alpha
\end{bmatrix}
\leftarrow \begin{bmatrix} \bm{P}(\cdot, \mathsf{S}_i) & \bm{P}(\cdot, i) \end{bmatrix}^* \bm{A} \begin{bmatrix} \bm{P}(\cdot, \mathsf{S}_i) & \bm{P}(\cdot, i) \end{bmatrix}$}
\State{Solve positive-semidefinite system $\bm{M}\bm{x} = -\bm{v}$}
\State{$\hat{\bm{C}}(i, \mathsf{S}_i) \leftarrow \bm{x}^*$}
\State{$\hat{\bm{D}}(i, i) \leftarrow \alpha + \bm{x}^*\bm{v}$}
\EndFor
\end{algorithmic}
\end{algorithm}

\Cref{alg:vecchia} presents a conventional Vecchia approach that iteratively solves a linear system to determine each row vector $\hat{\bm{C}}(i, \mathsf{S}_i)$ and then takes an inner product to determine $\hat{\bm{D}}(i,i)$.
This algorithm requires $\mathcal{O}(s^2 n)$ entry lookups and $\mathcal{O}(s^3 n)$ additional arithmetic operations, where $s$ is an upper bound on the sparsity pattern: $|\mathsf{S}_i| \leq s$ for each $i = 1, \ldots, n$.
However, the Vecchia approximation can be generated more quickly when the sparsity index sets $\mathsf{S}_i$ are highly overlapping.
Indeed, \cref{sec:main_result} will describe a specialized Vecchia construction with a smaller cost of $\mathcal{O}(s n)$ entry lookups and $\mathcal{O}(s^2 n)$ arithmetic operations.

The design choices in the Vecchia approximation are the permutation $\bm{P}$ and the sparsity pattern $(\mathsf{S}_i)_{i=1}^n$.
Historically, researchers chose the sparsity pattern using nearest neighbor basis vectors in the $\bm{A}$-weighted distance \cite{sun2016statistically},
and they generated the permutation matrix by recursively choosing the farthest-away basis vector in the same distance \cite{guinness2018permutation}.
More recently, Huan et al. \cite{huan2023sparse} proposed optimizing each sparsity index set using the orthogonal matching pursuit algorithm \cite{tropp2004greed}.
See \cref{sec:adding} for more discussion.

\subsection{Partial Cholesky + Vecchia = Vecchia} \label{sec:main_result}

\begin{algorithm}[t]
\caption{Partial Cholesky + Vecchia approximation \cite{zhao2024adaptive,cai2025posterior}} \label{alg:hybrid}
\begin{algorithmic}[1]
\Require{Permuation matrix $\bm{P} \in \{0, 1\}^{n \times n}$; Cholesky rank $r$; Vecchia sparsity pattern $(\mathsf{Q}_i)_{i=1}^n$; positive-semidefinite matrix $\bm{A} \in \mathbb{C}^{n \times n}$ with entry-wise access}
\Ensure{Vecchia approximation $\hat{\bm{A}} = \bm{P} \hat{\bm{C}}^{-1} \hat{\bm{D}} \hat{\bm{C}}^{-*} \bm{P}^*$ in factored form}
\State{Generate partial Cholesky approximation with permutation $\bm{P}$ and approximation rank $r$.
\begin{equation*}
    \hat{\bm{A}}_{\rm part}
    = \bm{P} 
    \begin{bmatrix}
    \hat{\bm{L}}_{11} & \\
    \hat{\bm{L}}_{21} & \mathbf{I}
    \end{bmatrix} 
    \begin{bmatrix} \hat{\bm{D}}_{11} & \\
    & \mathbf{0} \end{bmatrix}
    \begin{bmatrix}
    \hat{\bm{L}}_{11} & \\
    \hat{\bm{L}}_{21} & \mathbf{I}
    \end{bmatrix}^*
    \bm{P}^*,
\end{equation*}
where the matrices are partitioned into the first $r$ and last $n-r$ entries.}
\State{Let $\bm{R} =  \bm{A} - \hat{\bm{A}}_{\rm part}$ be the residual from the partial Cholesky approximation.
Generate Vecchia approximation with permutation $\bm{P}$ and sparsity pattern $(\mathsf{Q}_i)_{i=1}^n$.
\begin{equation*}
    \hat{\bm{A}}_{\rm res} = \bm{P} \begin{bmatrix}
        \mathbf{I} & \\
        \bm{0} & \hat{\bm{C}}_{22}
    \end{bmatrix}^{-1} 
    \begin{bmatrix}
        \bm{0} & \\
        & \hat{\bm{D}}_{22}
    \end{bmatrix}
    \begin{bmatrix}
        \mathbf{I} & \\
        \bm{0} & \hat{\bm{C}}_{22}
    \end{bmatrix}^{-*} \bm{P}^*.
\end{equation*}
}
\State{Define $\hat{\bm{C}}$ and $\hat{\bm{D}}$ according to
\begin{equation*}
    \hat{\bm{C}} = \begin{bmatrix}
        \hat{\bm{L}}_{11}^{-1} & \\
        -\hat{\bm{C}}_{22} \hat{\bm{L}}_{21} \hat{\bm{L}}_{11}^{-1} & \hat{\bm{C}}_{22}
    \end{bmatrix}
    \quad \text{and} \quad
    \hat{\bm{D}} = \begin{bmatrix}
        \hat{\bm{D}}_{11} & \\
        & \hat{\bm{D}}_{22}
    \end{bmatrix}.
\end{equation*}
}
\end{algorithmic}
\end{algorithm}

This paper analyzes a hybrid approximation \cite{zhao2024adaptive,cai2025posterior} combining the partial pivoted Cholesky and Vecchia approximations.
See \cref{alg:hybrid} for pseudocode.
We will prove that this approximation is equivalent to a Vecchia approximation with an augmented sparsity pattern.

\begin{theorem}[Partial Cholesky $+$ Vecchia = Vecchia] \label{thm:everything_vecchia}
    Given a target positive-semidefinite matrix $\bm{A} \in \mathbb{C}^{n \times n}$, consider the following two-part approximation.
    \begin{enumerate}
        \item Generate a partial Cholesky approximation of $\bm{A}$ with permutation $\bm{P}$ and approximation rank $r$. Call it $\hat{\bm{A}}_{\rm part}$.
        \item Generate a Vecchia approximation of the residual $\bm{R} = \bm{A} - \hat{\bm{A}}_{\rm part}$ with permutation $\bm{P}$ and sparsity pattern $(\mathsf{Q}_i)_{i=1}^n$. Call it $\hat{\bm{A}}_{\rm res}$.
    \end{enumerate}
    Then $\hat{\bm{A}}_{\rm part} + \hat{\bm{A}}_{\rm res}$ can be rewritten as a Vecchia approximation of $\bm{A}$ with permutation $\bm{P}$ and an augmented sparsity pattern $\mathsf{S}_i = \bigl(\{1, \ldots, r\} \cup \mathsf{Q}_i\bigr) \cap \{1, \ldots, i-1\}$.
\end{theorem}
\begin{proof}
    For notational simplicity, we assume the permutation is $\bm{P} = \mathbf{I}$.
    Otherwise, we can permute the indices of $\bm{A}$ and apply the proof to the permuted matrix.
    We carry out the proof in two steps.

    \underline{Step 1: $\hat{\bm{A}}$ is a sparse inverse Cholesky approximation.}
    We start by writing
    \begin{equation*}
        \hat{\bm{A}}_{\rm part} 
        = \begin{bmatrix}
        \hat{\bm{L}}_{11} & \\
        \hat{\bm{L}}_{21} & \mathbf{I}
        \end{bmatrix} 
        \begin{bmatrix} \hat{\bm{D}}_{11} & \\
        & 
        & \mathbf{0} \end{bmatrix}
        \begin{bmatrix}
        \hat{\bm{L}}_{11} & \\
        \hat{\bm{L}}_{21} & \mathbf{I}
        \end{bmatrix}^*
        \quad \text{and} \quad
        \bm{R}
        = \begin{bmatrix} \bm{0} & \bm{0} \\
        \bm{0} & \bm{R}_{22}
        \end{bmatrix},
    \end{equation*}
    where we have partitioned the matrices into the first $r$ and last $n-r$ entries.
    By \cref{def:vecchia_approximation}, the first $r$ entries in the Vecchia diagonal factor are zero, and it follows
    \begin{align*}
        \hat{\bm{A}}_{\rm res} &= \begin{bmatrix}
            \hat{\bm{C}}_{11} & \\
            \hat{\bm{C}}_{21} & \hat{\bm{C}}_{22}
        \end{bmatrix}^{-1} 
        \begin{bmatrix}
            \bm{0} & \\
            & \hat{\bm{D}}_{22}
        \end{bmatrix}
        \begin{bmatrix}
            \hat{\bm{C}}_{11} & \\
            \hat{\bm{C}}_{21} & \hat{\bm{C}}_{22}
        \end{bmatrix}^{-*} \\
        &= \begin{bmatrix}
            \hat{\bm{C}}_{11}^{-1} & \\
            -\hat{\bm{C}}_{22}^{-1} \hat{\bm{C}}_{21} \hat{\bm{C}}_{11}^{-1} & \hat{\bm{C}}_{22}^{-1}
        \end{bmatrix}
        \begin{bmatrix}
            \bm{0} & \\
            & \hat{\bm{D}}_{22}
        \end{bmatrix}
        \begin{bmatrix}
            \hat{\bm{C}}_{11}^{-1} & \\
            -\hat{\bm{C}}_{22}^{-1} \hat{\bm{C}}_{21} \hat{\bm{C}}_{11}^{-1} & \hat{\bm{C}}_{22}^{-1}
        \end{bmatrix}^* \\
        &= \begin{bmatrix}
            \mathbf{I} & \\
            \bm{0} & \hat{\bm{C}}_{22}
        \end{bmatrix}^{-1} 
        \begin{bmatrix}
            \bm{0} & \\
            & \hat{\bm{D}}_{22}
        \end{bmatrix}
        \begin{bmatrix}
            \mathbf{I} & \\
            \bm{0} & \hat{\bm{C}}_{22}
        \end{bmatrix}^{-*}.
    \end{align*}
    We have simplified the formula for the residual Vecchia approximation by making the first $r$ columns in the lower triangular factor standard basis vectors.
    Consequently, we can write
    \begin{align*}
    \hat{\bm{A}} &= \hat{\bm{A}}_{\rm part} + \hat{\bm{A}}_{\rm res} \\
    &= \begin{bmatrix}
    \hat{\bm{L}}_{11} & \\
    \hat{\bm{L}}_{21} & \hat{\bm{C}}_{22}^{-1}
    \end{bmatrix} 
    \begin{bmatrix} \hat{\bm{D}}_{11} & \\
    & 
    & \hat{\bm{D}}_{22} \end{bmatrix}
    \begin{bmatrix}
    \hat{\bm{L}}_{11} & \\
    \hat{\bm{L}}_{21} & \hat{\bm{C}}_{22}^{-1}
    \end{bmatrix}^* \\
    &= {\underbrace{\begin{bmatrix}
    \hat{\bm{L}}_{11}^{-1} & \\
    -\hat{\bm{C}}_{22} \hat{\bm{L}}_{21} \hat{\bm{L}}_{11}^{-1} & \hat{\bm{C}}_{22}
    \end{bmatrix}}_{\hat{\bm{C}}}}^{-1}
    \underbrace{\begin{bmatrix} \hat{\bm{D}}_{11} & \\
    & 
    & \hat{\bm{D}}_{22} \end{bmatrix}}_{\hat{\bm{D}}}
    {\underbrace{\begin{bmatrix}
    \hat{\bm{L}}_{11}^{-1} & \\
    -\hat{\bm{C}}_{22} \hat{\bm{L}}_{21} \hat{\bm{L}}_{11}^{-1} & \hat{\bm{C}}_{22}
    \end{bmatrix}}_{\hat{\bm{C}}}}^{-*}.
    \end{align*}
    This derivation shows that
    \textbf{the sparsity pattern of $\hat{\bm{A}}$ is the union of the sparsity patterns of $\hat{\bm{A}}_{\rm part}$ and $\hat{\bm{A}}_{\rm res}$.}
    See \cref{fig:crazy} for an illustration.

    \begin{figure}[t]
    \centering
    \begin{tikzpicture}[scale=0.55]
    
    \tikzset{
      matrix box/.style={draw=black, thin},
      nz/.style={fill opacity=1.0},
    }
    
    \begin{scope}[shift={(0,0)}]
      \draw[matrix box] (0,0) rectangle (3,3);
      \fill[nz, fill=gray!50] (1.0,2.5) rectangle (1.5,3.0); 
      \fill[nz, fill=gray!50] (2.5,2.5) rectangle (3.0,3.0); 
      \fill[nz, fill=gray!50] (1.0,2.0) rectangle (1.5,2.5); 
      \fill[nz, fill=gray!50] (2.5,2.0) rectangle (3.0,2.5); 
      \fill[nz, fill=gray!50] (0.0,1.5) rectangle (0.5,2.0); 
      \fill[nz, fill=gray!50] (0.5,1.5) rectangle (1.0,2.0); 
      \fill[nz, fill=gray!50] (1.0,1.5) rectangle (1.5,2.0); 
      \fill[nz, fill=gray!50] (1.5,1.5) rectangle (2.0,2.0); 
      \fill[nz, fill=gray!50] (2.0,1.5) rectangle (2.5,2.0); 
      \fill[nz, fill=gray!50] (2.5,1.5) rectangle (3.0,2.0); 
      \fill[nz, fill=gray!50] (1.0,1.0) rectangle (1.5,1.5); 
      \fill[nz, fill=gray!50] (2.5,1.0) rectangle (3.0,1.5); 
      \fill[nz, fill=gray!50] (1.0,0.5) rectangle (1.5,1.0); 
      \fill[nz, fill=gray!50] (2.5,0.5) rectangle (3.0,1.0); 
      \fill[nz, fill=gray!50] (0.0,0.0) rectangle (0.5,0.5); 
      \fill[nz, fill=gray!50] (0.5,0.0) rectangle (1.0,0.5); 
      \fill[nz, fill=gray!50] (1.0,0.0) rectangle (1.5,0.5); 
      \fill[nz, fill=gray!50] (1.5,0.0) rectangle (2.0,0.5); 
      \fill[nz, fill=gray!50] (2.0,0.0) rectangle (2.5,0.5); 
      \fill[nz, fill=gray!50] (2.5,0.0) rectangle (3.0,0.5); 
      \node at (1.5,-0.6) {$\bm{A}$};
    \end{scope}
    
    \node at (3.75,1.5) {$\rightarrow$};
    
    \begin{scope}[shift={(4.5,0)}]
      \draw[matrix box] (0,0) rectangle (3,3);
      \fill[nz, fill=cyan!50] (0.0,2.5) rectangle (0.5,3.0); 
      \fill[nz, fill=cyan!50] (0.0,2.0) rectangle (0.5,2.5); 
      \fill[nz, fill=cyan!50] (0.5,2.0) rectangle (1.0,2.5); 
      \fill[nz, fill=cyan!50] (0.0,1.5) rectangle (0.5,2.0); 
      \fill[nz, fill=cyan!50] (0.5,1.5) rectangle (1.0,2.0); 
      \fill[nz, fill=cyan!50] (1.0,1.5) rectangle (1.5,2.0); 
      \fill[nz, fill=cyan!50] (0.0,1.0) rectangle (0.5,1.5); 
      \fill[nz, fill=cyan!50] (0.5,1.0) rectangle (1.0,1.5); 
      \fill[nz, fill=cyan!50] (1.5,1.0) rectangle (2.0,1.5); 
      \fill[nz, fill=cyan!50] (0.0,0.5) rectangle (0.5,1.0); 
      \fill[nz, fill=cyan!50] (0.5,0.5) rectangle (1.0,1.0); 
      \fill[nz, fill=cyan!50] (2.0,0.5) rectangle (2.5,1.0); 
      \fill[nz, fill=cyan!50] (0.0,0.0) rectangle (0.5,0.5); 
      \fill[nz, fill=cyan!50] (0.5,0.0) rectangle (1.0,0.5); 
      \fill[nz, fill=cyan!50] (2.5,0.0) rectangle (3.0,0.5); 
      \node at (1.5,-0.6) {$\hat{\bm{L}}_{\rm part}$};
      \node at (3.3,3.3) {$\phantom{{-1}}$};
    \end{scope}
    
    \begin{scope}[shift={(8.5,0)}]
      \draw[matrix box] (0,0) rectangle (3,3);
      \fill[nz, fill=teal!50] (0.0,2.5) rectangle (0.5,3.0); 
      \fill[nz, fill=teal!50] (0.5,2.0) rectangle (1.0,2.5); 
      \node at (1.5,-0.6) {$\hat{\bm{D}}_{\rm part}$};
    \end{scope}
    
    \begin{scope}[shift={(12.0,0)}]
      \draw[matrix box] (0,0) rectangle (3,3);
      \fill[nz, fill=cyan!50] (0.0,2.5) rectangle (0.5,3.0); 
      \fill[nz, fill=cyan!50] (0.5,2.5) rectangle (1.0,3.0); 
      \fill[nz, fill=cyan!50] (1.0,2.5) rectangle (1.5,3.0); 
      \fill[nz, fill=cyan!50] (1.5,2.5) rectangle (2.0,3.0); 
      \fill[nz, fill=cyan!50] (2.0,2.5) rectangle (2.5,3.0); 
      \fill[nz, fill=cyan!50] (2.5,2.5) rectangle (3.0,3.0); 
      \fill[nz, fill=cyan!50] (0.5,2.0) rectangle (1.0,2.5); 
      \fill[nz, fill=cyan!50] (1.0,2.0) rectangle (1.5,2.5); 
      \fill[nz, fill=cyan!50] (1.5,2.0) rectangle (2.0,2.5); 
      \fill[nz, fill=cyan!50] (2.0,2.0) rectangle (2.5,2.5); 
      \fill[nz, fill=cyan!50] (2.5,2.0) rectangle (3.0,2.5); 
      \fill[nz, fill=cyan!50] (1.0,1.5) rectangle (1.5,2.0); 
      \fill[nz, fill=cyan!50] (1.5,1.0) rectangle (2.0,1.5); 
      \fill[nz, fill=cyan!50] (2.0,0.5) rectangle (2.5,1.0); 
      \fill[nz, fill=cyan!50] (2.5,0.0) rectangle (3.0,0.5); 
      \node at (1.5,-0.6) {$\hat{\bm{L}}_{\rm part}^*$};
      \node at (3.3,3.3) {$\phantom{{-1}}$};
    \end{scope}
    
    \end{tikzpicture}
    
    \begin{tikzpicture}[scale=0.55]
    
    \tikzset{
      matrix box/.style={draw=black, thin},
      nz/.style={fill opacity=1.0},
    }
    
    \begin{scope}[shift={(0,0)}]
      \draw[matrix box] (0,0) rectangle (3,3);
      \fill[nz, fill=gray!50] (0.0,2.5) rectangle (0.5,3.0); 
      \fill[nz, fill=gray!50] (0.5,2.5) rectangle (1.0,3.0); 
      \fill[nz, fill=gray!50] (0.0,2.0) rectangle (0.5,2.5); 
      \fill[nz, fill=gray!50] (0.5,2.0) rectangle (1.0,2.5); 
      \fill[nz, fill=gray!50] (1.5,1.0) rectangle (2.0,1.5); 
      \fill[nz, fill=gray!50] (2.0,1.0) rectangle (2.5,1.5); 
      \fill[nz, fill=gray!50] (1.5,0.5) rectangle (2.0,1.0); 
      \fill[nz, fill=gray!50] (2.0,0.5) rectangle (2.5,1.0); 
      \node at (1.5,-0.6) {$\bm{R}$};
    \end{scope}
    
    \node at (3.75,1.5) {$\rightarrow$};
    
    \begin{scope}[shift={(4.5,0)}]
      \draw[matrix box] (0,0) rectangle (3,3);
      \fill[nz, fill=violet!50] (0.0,2.5) rectangle (0.5,3.0); 
      \fill[nz, fill=violet!50] (0.5,2.0) rectangle (1.0,2.5); 
      \fill[nz, fill=violet!50] (1.0,1.5) rectangle (1.5,2.0); 
      \fill[nz, fill=violet!50] (1.0,1.0) rectangle (1.5,1.5); 
      \fill[nz, fill=violet!50] (1.5,1.0) rectangle (2.0,1.5); 
      \fill[nz, fill=violet!50] (2.0,0.5) rectangle (2.5,1.0); 
      \fill[nz, fill=violet!50] (2.0,0.0) rectangle (2.5,0.5); 
      \fill[nz, fill=violet!50] (2.5,0.0) rectangle (3.0,0.5); 
      \node at (1.5,-0.6) {$\hat{\bm{C}}_{\rm res}$};
      \node at (3.3,3.3) {${-1}$};
    \end{scope}
    
    \begin{scope}[shift={(8.5,0)}]
      \draw[matrix box] (0,0) rectangle (3,3);
      \fill[nz, fill=teal!50] (1.0,1.5) rectangle (1.5,2.0); 
      \fill[nz, fill=teal!50] (1.5,1.0) rectangle (2.0,1.5); 
      \fill[nz, fill=teal!50] (2.0,0.5) rectangle (2.5,1.0); 
      \fill[nz, fill=teal!50] (2.5,0.0) rectangle (3.0,0.5); 
      \node at (1.5,-0.6) {$\hat{\bm{D}}_{\rm res}$};
    \end{scope}
    
    \begin{scope}[shift={(12.0,0)}]
      \draw[matrix box] (0,0) rectangle (3,3);
      \fill[nz, fill=violet!50] (0.0,2.5) rectangle (0.5,3.0); 
      \fill[nz, fill=violet!50] (0.5,2.0) rectangle (1.0,2.5); 
      \fill[nz, fill=violet!50] (1.0,1.5) rectangle (1.5,2.0); 
      \fill[nz, fill=violet!50] (1.5,1.5) rectangle (2.0,2.0); 
      \fill[nz, fill=violet!50] (1.5,1.0) rectangle (2.0,1.5); 
      \fill[nz, fill=violet!50] (2.0,0.5) rectangle (2.5,1.0); 
      \fill[nz, fill=violet!50] (2.5,0.5) rectangle (3.0,1.0); 
      \fill[nz, fill=violet!50] (2.5,0.0) rectangle (3.0,0.5); 
      \node at (1.5,-0.6) {$\hat{\bm{C}}_{\rm res}^*$};
      \node at (3.3,3.3) {${-1}$};
    \end{scope}
    
    \end{tikzpicture}
    
    \begin{tikzpicture}[scale=0.55]
    
    \tikzset{
      matrix box/.style={draw=black, thin},
      nz/.style={fill opacity=1.0},
    }
    
    \begin{scope}[shift={(0,0)}]
        \node at (1.2,2.3) {Partial};
        \node at (1.5,1.5) {Cholesky};
        \node at (1.6,0.8) {+ Vecchia};
    \end{scope}
    
    \node at (3.75,1.5) {$\rightarrow$};
    
    \begin{scope}[shift={(4.5,0)}]
      \draw[matrix box] (0,0) rectangle (3,3);
      \fill[nz, fill=violet!50] (0.0,2.5) rectangle (0.5,3.0); 
      \fill[nz, fill=violet!50] (0.0,2.0) rectangle (0.5,2.5); 
      \fill[nz, fill=violet!50] (0.5,2.0) rectangle (1.0,2.5); 
      \fill[nz, fill=violet!50] (0.0,1.5) rectangle (0.5,2.0); 
      \fill[nz, fill=violet!50] (0.5,1.5) rectangle (1.0,2.0); 
      \fill[nz, fill=violet!50] (1.0,1.5) rectangle (1.5,2.0); 
      \fill[nz, fill=violet!50] (0.0,1.0) rectangle (0.5,1.5); 
      \fill[nz, fill=violet!50] (0.5,1.0) rectangle (1.0,1.5); 
      \fill[nz, fill=violet!50] (1.0,1.0) rectangle (1.5,1.5); 
      \fill[nz, fill=violet!50] (1.5,1.0) rectangle (2.0,1.5); 
      \fill[nz, fill=violet!50] (0.0,0.5) rectangle (0.5,1.0); 
      \fill[nz, fill=violet!50] (0.5,0.5) rectangle (1.0,1.0); 
      \fill[nz, fill=violet!50] (2.0,0.5) rectangle (2.5,1.0); 
      \fill[nz, fill=violet!50] (0.0,0.0) rectangle (0.5,0.5); 
      \fill[nz, fill=violet!50] (0.5,0.0) rectangle (1.0,0.5); 
      \fill[nz, fill=violet!50] (2.0,0.0) rectangle (2.5,0.5); 
      \fill[nz, fill=violet!50] (2.5,0.0) rectangle (3.0,0.5); 
      \node at (1.5,-0.6) {$\hat{\bm{C}}$};
      \node at (3.3,3.3) {${-1}$};
    \end{scope}
    
    \begin{scope}[shift={(8.5,0)}]
      \draw[matrix box] (0,0) rectangle (3,3);
      \fill[nz, fill=teal!50] (0.0,2.5) rectangle (0.5,3.0); 
      \fill[nz, fill=teal!50] (0.5,2.0) rectangle (1.0,2.5); 
      \fill[nz, fill=teal!50] (1.0,1.5) rectangle (1.5,2.0); 
      \fill[nz, fill=teal!50] (1.5,1.0) rectangle (2.0,1.5); 
      \fill[nz, fill=teal!50] (2.0,0.5) rectangle (2.5,1.0); 
      \fill[nz, fill=teal!50] (2.5,0.0) rectangle (3.0,0.5); 
      \node at (1.5,-0.6) {$\hat{\bm{D}}$};
    \end{scope}
    
    \begin{scope}[shift={(12.0,0)}]
      \draw[matrix box] (0,0) rectangle (3,3);
      \fill[nz, fill=violet!50] (0.0,2.5) rectangle (0.5,3.0); 
      \fill[nz, fill=violet!50] (0.5,2.5) rectangle (1.0,3.0); 
      \fill[nz, fill=violet!50] (1.0,2.5) rectangle (1.5,3.0); 
      \fill[nz, fill=violet!50] (1.5,2.5) rectangle (2.0,3.0); 
      \fill[nz, fill=violet!50] (2.0,2.5) rectangle (2.5,3.0); 
      \fill[nz, fill=violet!50] (2.5,2.5) rectangle (3.0,3.0); 
      \fill[nz, fill=violet!50] (0.5,2.0) rectangle (1.0,2.5); 
      \fill[nz, fill=violet!50] (1.0,2.0) rectangle (1.5,2.5); 
      \fill[nz, fill=violet!50] (1.5,2.0) rectangle (2.0,2.5); 
      \fill[nz, fill=violet!50] (2.0,2.0) rectangle (2.5,2.5); 
      \fill[nz, fill=violet!50] (2.5,2.0) rectangle (3.0,2.5); 
      \fill[nz, fill=violet!50] (1.0,1.5) rectangle (1.5,2.0); 
      \fill[nz, fill=violet!50] (1.5,1.5) rectangle (2.0,2.0); 
      \fill[nz, fill=violet!50] (1.5,1.0) rectangle (2.0,1.5); 
      \fill[nz, fill=violet!50] (2.0,0.5) rectangle (2.5,1.0); 
      \fill[nz, fill=violet!50] (2.5,0.5) rectangle (3.0,1.0); 
      \fill[nz, fill=violet!50] (2.5,0.0) rectangle (3.0,0.5); 
      \node at (1.5,-0.6) {$\hat{\bm{C}}^*$};
      \node at (3.3,3.3) {${-1}$};
    \end{scope}
    
    \end{tikzpicture}
    
    \caption{First row: partial Cholesky accesses the gray entries of $\bm{A}$ to generate an approximation $\hat{\bm{A}}_{\rm part} = \hat{\bm{L}}_{\rm part} \hat{\bm{D}}_{\rm part} \hat{\bm{L}}_{\rm part}^*$.
    Second row: Vecchia accesses the gray entries of $\bm{R} = \bm{A} - \hat{\bm{A}}_{\rm part}$ to generate an approximation $\hat{\bm{A}}_{\rm res} = \hat{\bm{C}}_{\rm res}^{-1} \hat{\bm{D}}_{\rm res} \hat{\bm{C}}_{\rm res}^{-*}$. 
    Third row: partial Cholesky + Vecchia  yields an improved approximation $\hat{\bm{A}} = \hat{\bm{C}}^{-1} \hat{\bm{D}} \hat{\bm{C}}^{-*}$.
    \textbf{The sparsity pattern in row 3 is the union of the sparsity patterns in rows 1 and 2.}}
    \label{fig:crazy}
    \end{figure}
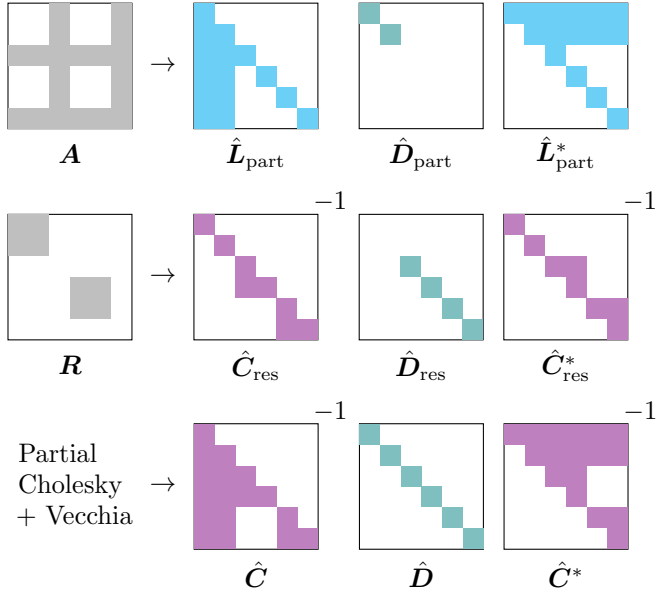

    \underline{Step 2: $\hat{\bm{A}}$ is a Vecchia approximation.}
    To show $\hat{\bm{A}}$ is a Vecchia approximation, we need to verify the following equalities for $i = 1, \ldots, n$:
    \begin{equation}
    \label{eq:to_confirm}
        (\hat{\bm{C}} \bm{A})(i, \mathsf{S}_i) = 0
        \quad \text{and}
        \quad (\hat{\bm{C}} \bm{A})(i, i) = \hat{\bm{D}}(i,i).
    \end{equation}
    First consider the indices $i \leq r$.
    Since $\bm{A}$ and $\hat{\bm{A}}_{\rm part}$ have the same first $r$ rows, we can write
    \begin{align*}
        (\hat{\bm{C}} \bm{A})(i, \cdot)
        &= (\hat{\bm{C}} \hat{\bm{A}}_{\rm part})(i, \cdot) \\
        &= (\hat{\bm{L}}_{11}^{-1} \hat{\bm{L}}_{11} \hat{\bm{D}}_{11})(i, \cdot) \begin{bmatrix}
            \hat{\bm{L}}_{11}^* & \hat{\bm{L}}_{21}^*
        \end{bmatrix} \\
        &= \hat{\bm{D}}(i,i) \begin{bmatrix}
            \hat{\bm{L}}_{11}(\cdot, i)^* & \hat{\bm{L}}_{21}(\cdot, i)^*
        \end{bmatrix}.
    \end{align*}
    Since the first $i - 1$ entries of $\hat{\bm{L}}_{11}(\cdot, i)$ are zero and the $i$th entry is one, the first $i - 1$ entries of $(\hat{\bm{C}} \bm{A})(i, \cdot)$ are zero and the $i$th entry is $\hat{\bm{D}}(i,i)$, confirming \cref{eq:to_confirm}.
    Next consider the indices $i \geq r + 1$.
    We make the calculation
    \begin{align*}
        (\hat{\bm{C}} \bm{A})(i, \cdot) 
        &= \hat{\bm{C}}(i, \cdot) (\hat{\bm{A}}_{\rm part} + \bm{R})
        \\
        &= \hat{\bm{C}}_{22}(i-r, \cdot) \begin{bmatrix} -\hat{\bm{L}}_{21} \hat{\bm{L}}_{11}^{-1} & \mathbf{I} \end{bmatrix} \Biggl(\begin{bmatrix} \hat{\bm{L}}_{11} \\ \hat{\bm{L}}_{21} \end{bmatrix} \hat{\bm{D}}_{11} \begin{bmatrix} \hat{\bm{L}}_{11}^* \hat{\bm{L}}_{21}^* \end{bmatrix} + \begin{bmatrix} \bm{0} & \bm{0} \\
        \bm{0} & \bm{R}_{22} \end{bmatrix}\Biggr) \\
        &= \hat{\bm{C}}_{22}(i-r, \cdot) \begin{bmatrix} \bm{0} & \bm{R}_{22} \end{bmatrix} \\
        &= \begin{bmatrix} \hat{\bm{C}}_{21}(i-r, \cdot) & \hat{\bm{C}}_{22}(i-r, \cdot) \end{bmatrix} \bm{R}.
    \end{align*}
    Since $\begin{bmatrix} \hat{\bm{C}}_{21}(i-r, \cdot) & \hat{\bm{C}}_{22}(i-r, \cdot) \end{bmatrix}$ is the $i$th row vector generated by the Vecchia approximation of $\bm{R}$, we confirm \cref{eq:to_confirm} and complete the proof.
\end{proof}

\subsection{Implications}
\label{sec:implications}

\Cref{thm:everything_vecchia} demonstrates that the partial Cholesky + Vecchia approach, which appeared in the previous papers \cite{zhao2024adaptive,cai2025posterior}, is secretly constructing a Vecchia approximation with the first $r$ indices included in the sparsity pattern.
We list two significant implications below.

First, because of the optimality of the Vecchia approximation (\cref{thm:optimality}), the partial Cholesky + Vecchia approach is \emph{theoretically optimal} in the sense of the Kaporin condition number.
Moreover, the partial Cholesky + Vecchia approach can be understood as a toolbox for applying Vecchia approximations in practice.
In the simplest case, we can build a partial Cholesky + Vecchia approximation with a trivial sparsity pattern: $\mathsf{Q}_i = \emptyset$ for $i = 1, \ldots, n$.
This approach reduces to a \emph{partial Cholesky + diagonal} approximation.
\begin{equation*}
    \hat{\bm{A}}
    = \hat{\bm{A}}_{\rm part} + \hat{\bm{A}}_{\rm res} 
    = \hat{\bm{A}}_{\rm part} + \operatorname{diag}(\bm{A} - \hat{\bm{A}}_{\rm part}).
\end{equation*}
We can also build an approximation with more nonzero entries included in the Vecchia residual component, thus improving the Kaporin condition number and improving the accuracy of linear algebra calculations.

Second, the partial Cholesky + Vecchia approach is a \emph{computationally efficient} way to construct a Vecchia approximation.
Suppose the partial Cholesky component has rank $r$, and the Vecchia residual component has bounded sparsity 
\begin{equation*}
\max_{1 \leq i \leq n} |\mathsf{Q}_i| = \mathcal{O}(r^{1/2}).
\end{equation*}
Then the partial Cholesky + Vecchia approximation can be constructed with $\mathcal{O}(r n)$ entry lookups and $\mathcal{O}(r^2 n)$ operations, which is much smaller than the conventional Vecchia construction cost of $\mathcal{O}(r^2 n)$ entry lookups and $\mathcal{O}(r^3 n)$ operations using \cref{alg:vecchia}.

\section{Kaporin optimality theory} \label{sec:optimality}

In this section, we ask and answer, ``In what way is the Vecchia approximation optimal?''
Vecchia optimality theory was previously developed in the papers \cite{vecchia1988estimation,kaporin1994new,axelsson2000sublinear,yeremin2000factorized,schafer2021sparse}.
Here, we review this optimality theory and push it in new directions while focusing specifically on the Kaporin condition number (\cref{def:kaporin}).
\Cref{vecchia_optimality} shows that the Vecchia approximation optimizes the Kaporin condition number.
Then \cref{sec:linear,sec:determinant} explore the implications for linear solves and determinant calculations, respectively.

\subsection{Vecchia optimality theorem} \label{vecchia_optimality}

Kaporin \cite[App.~A.3]{kaporin1994new} proved that the Vecchia approximation optimizes $\kappa_{\rm Kap}$ for any given sparsity pattern, assuming the target matrix $\bm{A}$ is strictly positive-definite.
The next result extends Kaporin's optimality theorem to any positive-semidefinite target matrix $\bm{A}$.
The proof is slightly long (three pages) and is presented in \cref{app:vecchia_proof}.

\begin{theorem}[Optimality of Vecchia]
\label{thm:optimality}
    For any positive-semidefinite matrix $\bm{A} \in \mathbb{C}^{n \times n}$, the Vecchia approximation $\hat{\bm{A}} = \bm{P} \hat{\bm{C}}^{-1} \hat{\bm{D}} \hat{\bm{C}}^{-*} \bm{P}^*$ is the inverse Cholesky approximation with permutation $\bm{P}$ and sparsity pattern $(\mathsf{S}_i)_{i=1}^n$ that achieves the smallest possible Kaporin condition number.
    If the Vecchia approximation has the same range as $\bm{A}$, the Kaporin condition number is
    \begin{equation}
    \label{eq:kaporin}
        \kappa_{\rm Kap} = \prod_{d_{\tilde{\bm{A}}}(\bm{e}_i, \operatorname{span}\{\bm{e}_j\}_{j < i}) > 0}
        \frac{d_{\tilde{\bm{A}}}\bigl(\bm{e}_i, \operatorname{span}\{\bm{e}_j\}_{j \in \mathsf{S}_i} \bigr)^2}{d_{\tilde{\bm{A}}}\bigl(\bm{e}_i, \operatorname{span}\{\bm{e}_j\}_{j < i} \bigr)^2},
        \qquad \text{for } \tilde{\bm{A}} = \bm{P}^* \bm{A} \bm{P}.
    \end{equation}
\end{theorem}

\Cref{thm:optimality} has three important implications.
First, suppose there is an \emph{exact} inverse Cholesky decomposition with the given sparsity pattern $(\mathsf{S}_i)_{i=1}^n$.
Then, the Vecchia approximation achieves the best possible Kaporin condition number $\kappa_{\rm Kap} = 1$.
In this case, the Vecchia definition ensures that $\hat{\bm{A}} = \bm{A}$ and perfect recovery occurs.

Second, when perfect recovery is not possible, \cref{thm:optimality} provides an explicit expression for the Kaporin condition number \cref{eq:kaporin}.
The proof shows how can write \cref{eq:kaporin} in several ways.
\begin{equation*}
    \kappa_{\rm Kap} = \frac{\operatorname{vol} (\hat{\bm{A}})}{\operatorname{vol} (\bm{A})}
    = \frac{\operatorname{vol} (\hat{\bm{D}})}{\operatorname{vol} (\bm{D})}
    = \prod_{\bm{D}(i,i) > 0} \frac{\hat{\bm{D}}(i,i)}{\bm{D}(i,i)},
\end{equation*}
where $\bm{A} = \bm{P} \bm{C}^{-1} \bm{D} \bm{C}^{-*} \bm{P}^*$ is an exact inverse Cholesky decomposition of the target matrix.
Further, the diagonal entries $\bm{D}(i,i)$ and $\hat{\bm{D}}(i,i)$ equal the square $\tilde{\bm{A}}$-weighted distances
\begin{equation*}
    \hat{\bm{D}}(i,i) = d_{\tilde{\bm{A}}}(\bm{e}_i, \operatorname{span}\{\bm{e}_j\}_{j \in \mathsf{S}_i})^2
    \quad \text{and} \quad
    \bm{D}(i,i) = d_{\tilde{\bm{A}}}(\bm{e}_i, \operatorname{span}\{\bm{e}_j\}_{j < i})^2.
\end{equation*}
The square $\tilde{\bm{A}}$-weighted distances are important because they motivate several optimization strategies for minimizing the Kaporin condition number that we will pursue in \cref{sec:optimization}.

Last, the proof of \cref{thm:optimality} shows how to avoid the bad case where $\kappa_{\rm Kap} = \infty$.
Specifically, for each ``bad'' index $i$ that satisfies
\begin{equation*}
    d_{\tilde{\bm{A}}}(\bm{e}_i, \operatorname{span}\{\bm{e}_j\}_{j < i}) = 0,
\end{equation*}
we must choose a sparsity index set $\mathsf{S}_i$ large enough to satisfy
\begin{equation}
\label{eq:bad}
    d_{\tilde{\bm{A}}}(\bm{e}_i, \operatorname{span}\{\bm{e}_j\}_{j \in \mathsf{S}_i}) = 0.
\end{equation}
If \cref{eq:bad} is violated,
then it becomes impossible to generate a sparse inverse Cholesky approximation with the same range as $\bm{A}$.
Thankfully, these bad indices only exist when $\bm{A}$ is rank-deficient --- in the full-rank case the Vecchia approximation always results in $\kappa_{\rm Kap} < \infty$.

\subsection{Linear solves} \label{sec:linear}
Now we show how to use a factored matrix approximation $\hat{\bm{A}}$ to accelerate linear algebra calculations, and the performance is bounded in terms of $\kappa_{\rm Kap}$.

In order to solve a linear system $\bm{A} \bm{x} = \bm{b}$ or a more general least-squares problem $\min_{\bm{x}} \lVert \bm{A} \bm{x} - \bm{b} \rVert^2$, we can make an initial guess $\bm{x}_0$ and refine our guess by calculating
\begin{equation*}
    \hat{\bm{x}} = \bm{x}_0 + \hat{\bm{A}}^+[\bm{b} - \bm{A} \bm{x}_0].
\end{equation*}
This calculation can be very fast.
In the case $\bm{x}_0 = \bm{0}$, we do not even need to examine the entries of $\bm{A}$ once; we only need to perform a linear solve with the factored approximation $\hat{\bm{A}}$.
We believe the following proposition gives a new error bound for the approximate direct solve, and the proof appears in \cref{sec:direct}.

\begin{proposition}[Approximate direct solver for linear systems] \label{prop:direct}
    If $\hat{\bm{A}}$ is normalized so that $\operatorname{tr}\bigl(\bm{A} \hat{\bm{A}}^+\bigr) = \operatorname{rank}(\bm{A})$, then $\hat{\bm{x}} = \bm{x}_0 + \hat{\bm{A}}^+[\bm{b} - \bm{A} \bm{x}_0]$ satisfies
    \begin{equation*}
        \frac{\lVert \hat{\bm{x}} - \bm{x}_{\star} \rVert_{\bm{A}}^2}{ \lVert \bm{x}_0 - \bm{x}_{\star} \rVert_{\bm{A}}^2}
        \leq 2 \operatorname{rank}(\bm{A}) \log(\kappa_{\rm Kap}).
    \end{equation*}
    Here, $\bm{x}_{\star} = \bm{A}^+ \bm{b}$ is the minimum--norm solution to $\min_{\bm{x}} \lVert \bm{A} \bm{x} - \bm{b} \rVert^2$.
\end{proposition}

The normalization assumption in \cref{prop:direct} always holds for the Vecchia approximation when $\kappa_{\rm Kap} < \infty$.
Yet unfortunately, \cref{prop:direct} does not guarantee any improvement over the initial guess $\bm{x}_0$ unless the Kaporin condition number is extremely small, $\log(\kappa_{\rm Kap}) < 1/\operatorname{rank}(\bm{A})$.

When the approximate direct solver fails to deliver an accurate solution to a linear system, we can instead solve the linear system iteratively using preconditioned conjugate gradient (PCG, \cite[Sec.~11.5]{golub2013matrix}).
Starting with an initial guess $\bm{x}_0$, PCG calculates the initial residual $\bm{r}_0 = \bm{b} - \bm{A} \bm{x}_0$ and initial search direction $\bm{d}_1 = \hat{\bm{A}}^+ \bm{r}_0$.
At each step $t = 1, 2, \ldots$, PCG updates the iterate and residual according to
\begin{equation*}
    \begin{cases}
        \bm{x}_t = \bm{x}_{t-1} + \alpha_t \bm{d}_t, \\
        \bm{r}_t = \bm{r}_{t-1} - \alpha_t \bm{A} \bm{d}_t.
    \end{cases}
\end{equation*}
The step size and search direction are selected according to
\begin{equation*}
    \alpha_t = \frac{\bm{r}_{t-1}^* \hat{\bm{A}}^+ \bm{r}_{t-1}}{\bm{d}_i^* \bm{A} \bm{d}_i}
    \quad \text{and} \quad
    \bm{d}_{t+1} = \hat{\bm{A}}^+ \bm{r}_t + \frac{\bm{r}_t^* \hat{\bm{A}}^+ \bm{r}_t}{\bm{r}_{t-1}^* \hat{\bm{A}}^+ \bm{r}_{t-1}} \bm{d}_t.
\end{equation*}
We can run PCG for any number of iterations, producing better and better estimates of $\bm{x}_{\star} = \bm{A}^+ \bm{b}$ with each iteration. 
PCG requires just one multiplication with $\bm{A}$ and one linear solve with $\hat{\bm{A}}$ per iteration.

Axelsson and Kaporin \cite[Thm.~4.3]{axelsson2000sublinear} bounded the error of PCG in terms of $\kappa_{\rm Kap}$, assuming an even number of iterations, and \cref{sec:pcg_proof} extends their proof to handle any even or odd number of iterations.
\begin{proposition}[Convergence of PCG \cite{axelsson2000sublinear}] \label{prop:pcg}
At each iteration $t \geq 0$, the iterates produced by preconditioned conjugate gradient satisfy
\begin{equation}
\label{eq:superlinear}
    \frac{\lVert \bm{x}_t - \bm{x}_{\star} \rVert_{\bm{A}}^2}{\lVert \bm{x}_0 - \bm{x}_{\star} \rVert_{\bm{A}}^2} \leq \biggl[\frac{3 \log(\kappa_{\rm Kap})}{t}\biggr]^t.
\end{equation}
Here, $\bm{x}_{\star} = \bm{A}^+ \bm{b}$ is the minimum--norm solution to $\min_{\bm{x}} \lVert \bm{A} \bm{x} - \bm{b} \rVert^2$.
\end{proposition}

\Cref{prop:pcg} guarantees \emph{superlinear} convergence, which is even stronger than the standard conjugate gradient linear convergence bounds, e.g., \cite[Eq.~11.3.27]{golub2013matrix}.
In large-scale applications, the superlinear convergence is mainly observed in the latter iterations when PCG has already achieved high accuracy \cite{axelsson2000sublinear}.

\subsection{Determinant calculations} \label{sec:determinant}

When $\bm{A}$ and $\hat{\bm{A}}$ are strictly positive-definite, we can use $\det (\hat{\bm{A}})$ as an estimator for $\det (\bm{A})$.
This computation is very fast for the Vecchia approximation $\hat{\bm{A}} = \bm{P} \hat{\bm{C}}^{-1} \hat{\bm{D}} \hat{\bm{C}}^{-*} \bm{P}^*$, because
\begin{equation*}
\det (\hat{\bm{A}}) = \det (\hat{\bm{D
}} )= \prod_{i=1}^n \hat{\bm{D}}(i,i).
\end{equation*}
The next proposition exactly describes the error in the determinant estimate.
\begin{proposition}[Approximate direct solver for determinants] \label{prop:det}
    If $\bm{A}$ and $\hat{\bm{A}}$ are strictly positive-definite and $\hat{\bm{A}}$ is normalized so that $\operatorname{tr}\bigl(\bm{A} \hat{\bm{A}}^{-1}\bigr) = n$, then
    \begin{equation*}
        \log\biggl(\frac{\det \hat{\bm{A}}}{\det \bm{A}}\biggr)
        = \log(\kappa_{\mathrm{Kap}}).
    \end{equation*}
\end{proposition}
\begin{proof}
    We rewrite the Kaporin condition number as
    \begin{equation*}
        \kappa_{\rm Kap} = \frac{\bigl(\frac{1}{n} \operatorname{tr} ( \bm{A} \hat{\bm{A}}^{-1})\bigr)^{n}}{\det (\bm{A} \hat{\bm{A}}^{-1})} 
        = \frac{1}{\det(\bm{A} \hat{\bm{A}}^{-1})} = \frac{\det(\hat{\bm{A}})}{\det(\bm{A})},
    \end{equation*}
    using the trace normalization $\operatorname{tr}\bigl(\bm{A} \hat{\bm{A}}^{-1}\bigr) = n$.
\end{proof}

The normalization assumption in \cref{prop:det} always holds for the Vecchia approximation if $\bm{A}$ is strictly positive-definite.
Therefore, the Vecchia determinant upper bounds the true determinant of $\bm{A}$.
Note, however, the Vecchia approximation does not always provide the sharpest upper bound for the determinant of $\bm{A}$ given the revealed entries:
the sharpest upper bound is given by the solution to the \emph{maximum entropy} problem; see \cite[pg.~161]{dempster1972}.

When the approximation $\det \bm{A} \approx \det \hat{\bm{A}}$ is not sufficiently accurate, we can refine this approximation with a multiplicative correction term.
To that end, we write
\begin{equation*}
    \log\biggl(\frac{\det \bm{A}}{\det\hat{\bm{A}}}\biggr)
    = \log\bigl(\det \bigl(\hat{\bm{A}}^{-1/2} \bm{A}\hat{\bm{A}}^{-1/2} \bigr) \bigr)
    = \operatorname{tr}\bigl(\log \bigl(\hat{\bm{A}}^{-1/2} \bm{A} \hat{\bm{A}}^{-1/2} \bigr) \bigr).
\end{equation*}
Here, the matrix logarithm is defined by taking the logarithm of a matrix's eigenvalues.
We can approximate the trace of the matrix logarithm by sampling independent length-$\sqrt{n}$ vectors with uniformly random directions, $\bm{u}_1, \ldots, \bm{u}_t$, and forming a \emph{stochastic trace estimator} \cite{ubaru2017fast}:
\begin{equation*}
    s_t = \frac{1}{t}\sum_{i=1}^t \bm{u}_i^* \log \bigl(\hat{\bm{A}}^{-1/2} \bm{A} \hat{\bm{A}}^{-1/2}\bigr) \bm{u}_i.
\end{equation*}
Assuming the matrix logarithm is calculated exactly, this estimator is unbiased and converges as $t \rightarrow \infty$.
The next new proposition bounds the mean square error, and the proof is in \cref{sec:det_proof}.
\begin{proposition}[Stochastic determinant estimation] \label{prop:stochastic_det}
    Suppose $\bm{A}$ and $\hat{\bm{A}}$ are strictly positive-definite $n \times n$ matrices with Kaporin condition number $\kappa_{\rm Kap} \leq {\rm e}^n$.
    Suppose $\bm{u}_1, \ldots, \bm{u}_t$ are independent vectors with length $\sqrt{n}$ and uniformly random directions, and set $s_t = \frac{1}{t}\sum_{i=1}^t \bm{u}_i^* \log \bigl(\hat{\bm{A}}^{-1/2} \bm{A} \hat{\bm{A}}^{-1/2}\bigr) \bm{u}_i$.
    Then
    \begin{equation}
    \label{eq:stochastic}
        \mathbb{E}\biggl|\log\biggl(
        \frac{{\rm e}^{s_t}\det \hat{\bm{A}}}{\det \bm{A}}
        \biggr)\biggr|^2
        \leq
        \frac{4 \log(\kappa_{\mathrm{Kap}})}{t}.
    \end{equation}
    if $\bm{u}_1, \ldots, \bm{u}_n$ are complex-valued. The same bound holds with $4$ replaced by $8$ if $\bm{u}_1, \ldots, \bm{u}_n$ are real-valued.
\end{proposition}
\Cref{prop:stochastic_det} shows how the stochastic determinant estimator achieves a smaller mean square error than the approximate direct determinant estimator as soon as the number of iterations is $t \geq 4/\log(\kappa_{\rm Kap})$.

In stochastic determinant estimation, we need to evaluate terms of the form 
\begin{equation*}
    \bm{u}_i^* \log(\bm{B}) \bm{u}_i,
    \qquad \text{where } \bm{B} = \hat{\bm{A}}^{-1/2} \bm{A} \hat{\bm{A}}^{-1/2}.
\end{equation*}
To evaluate each term, we can first generate an orthonormal basis $\bm{Q} \in \mathbb{C}^{n \times m}$ for the Krylov subspace
\begin{equation*}
    K_m(\bm{B}, \bm{u}_i)
    = \operatorname{span}\{\bm{u}_i, \bm{B} \bm{u}_i, \ldots, \bm{B}^m \bm{u}_i\}.
\end{equation*}
We can then return the Krylov-Ritz approximation \cite{frommer2016error}
\begin{equation*}
    \bm{u}_i^* \log(\bm{B}) \bm{u}_i
    \approx \bm{u}_i^* \bm{Q} \log(\bm{Q}^* \bm{B} \bm{Q}) \bm{Q}^* \bm{u}_i.
\end{equation*}
We calculate $\log(\bm{Q}^* \bm{B} \bm{Q})$ by directly applying the logarithm to the matrix eigenvalues.
The cost is then dominated by $m$ matrix--vector multiplications, requiring $\mathcal{O}(m n^2)$ arithmetic operations.
See \cref{sec:log} for a comparison of Krylov-Ritz approximations with different depth parameters $m \in \{100, 1000\}$.

\section{Optimization strategies} \label{sec:optimization}

This section investigates different sparsity patterns, with the goal of minimize the Kaporin condition number.
First we consider different pivot sets that can be used in the partial Cholesky + diagonal approximation (\cref{sec:partial_plus}).
Then we add nonzero elements to the Vecchia residual component to bring down the Kaporin condition number further
(\cref{sec:adding}).

\subsection{Partial Cholesky + diagonal} \label{sec:partial_plus}

The Kaporin condition number for the partial Cholesky + diagonal approximation, when finite, can be written as
\begin{equation}
\label{eq:rewrite}
    \kappa_{\rm Kap}(\mathsf{R}) = \prod_{d_{\bm{A}}(\bm{e}_i, \operatorname{span}\{\bm{e}_j\}_{j \in \mathsf{R}}) > 0}
    \frac{d_{\bm{A}}\bigl(\bm{e}_i, \operatorname{span}\{\bm{e}_j\}_{j \in \mathsf{R}} \bigr)^2}{d_{\bm{A}}\bigl(\bm{e}_i, \operatorname{span}\bigl(\{\bm{e}_j\}_{j < i} \cup \{\bm{e}_j\}_{j \in \mathsf{R}} \bigr)^2}.
\end{equation}
$\kappa_{\rm Kap}(\mathsf{R})$ depends only on the partial Cholesky pivot set $\mathsf{R} = \{u_1, \ldots, u_r\}$.
Below we consider several strategies for choosing a pivot set to minimize the Kaporin condition number.

\subsubsection{Adaptive search} \label{sec:adaptive_search}
We first introduce a direct optimization strategy that we call \emph{adaptive search}.
Adaptive search starts with an empty pivot set $\mathsf{R} = \emptyset$.
At each stage, this algorithm tries all available pivots and inducts a new pivot causing the maximum decrease in the Kaporin condition number.
\begin{flalign*}
    \qquad \qquad \mathsf{R} \leftarrow \mathsf{R} \cup \{i\},
    \quad \text{where } i \in \operatornamewithlimits{argmin}_{1 \leq j \leq n} \kappa_{\rm Kap}(\mathsf{R} \cup \{j\})
    && \text{[adaptive search]}. \qquad
\end{flalign*}
Adaptive search can lead to high-quality approximations, but it is expensive.
Each stage requires processing \emph{all} the entries in the matrix, and the cost of producing a cardinality-$r$ pivot set is thus $\mathcal{O}(r n^2)$ arithmetic operations.

\subsubsection{Adaptive sampling} \label{sec:adaptsampalgos}
We next consider a class of cheaper optimization strategies that we call \emph{adaptive sampling}.
Adaptive sampling methods are commonly used to select pivots for column-based matrix approximations, including partial Cholesky + Vecchia approximations \cite{zhao2024adaptive,cai2025posterior}.
Each adaptive sampling algorithm is based on a selection rule that prioritizes pivots ``far away'' from the already-selected pivots.
The most common pivot selection rules are defined as follows.
\begin{itemize}
\item In ``randomly pivoted Cholesky'' (RPC, \cite{chen2025randomly}), each new pivot is randomly sampled from the following adaptive probability distribution.
\begin{flalign*}
    \quad \mathsf{R} \leftarrow \mathsf{R} \cup \{i\},
    \quad \text{with } 
    \operatorname{prob}(i) = \frac{d_{\bm{A}}\bigl(\bm{e}_i, \operatorname{span}\{\bm{e}_j\}_{j \in \mathsf{R}}\bigr)^2}{\sum_{j=1}^n d_{\bm{A}}\bigl(\bm{e}_j, \operatorname{span}\{\bm{e}_j\}_{j \in \mathsf{R}} \bigr)^2}. && \mathrm{[RPC]} \quad
\end{flalign*}
\item 
In ``square distance sampling'' (SDS, \cite{arthur2007kmeans}), each new pivot is randomly sampled with slightly different probabilities.
\begin{flalign*}
    \quad \mathsf{R} \leftarrow \mathsf{R} \cup \{i\},
    \quad \text{with } \operatorname{prob}(i) = \frac{d_{\bm{A}}\bigl(\bm{e}_i, \{\bm{e}_j\}_{j \in \mathsf{R}}\bigr)^2}{\sum_{j = 1}^n d_{\bm{A}}\bigl(\bm{e}_j, \{\bm{e}_j\}_{j \in \mathsf{R}}\bigr)^2}.
    && \mathrm{[SDS]} \quad
\end{flalign*}
\item
In ``column pivoted Cholesky'' (CPC, \cite{higham1990analysis}), each new pivot maximizes the following distance.
\begin{flalign*}
    \quad \mathsf{R} \leftarrow \mathsf{R} \cup \{i\},
    \quad \text{where } i \in \operatornamewithlimits{argmax}_{1 \leq j \leq n} d_{\bm{A}}\bigl(\bm{e}_j, \operatorname{span}\{\bm{e}_j\}_{j \in \mathsf{R}} \bigr).
    && \mathrm{[CPC]} \quad
\end{flalign*}
\item
Last, in ``farthest point sampling'' (FPS, \cite{gonzalez1985clustering}), we apply a slightly different maximum distance rule.
\begin{flalign*}
    \quad \mathsf{R} \leftarrow \mathsf{R} \cup \{i\},
    \quad \text{where } i \in \operatornamewithlimits{argmax}_{1 \leq j \leq n} d_{\bm{A}}\bigl(\bm{e}_j, \{\bm{e}_j\}_{j \in \mathsf{R}} \bigr).
    && \mathrm{[FPS]} \quad 
\end{flalign*}
\end{itemize}
\noindent If there are ties in CPC or FPS, we can apply any tie-breaking rule such as uniformly random selection.

Each of the adaptive sampling algorithms is cheap to perform.
Recall that the partial Cholesky approximation has a construction cost of $\mathcal{O}(rn)$ entry lookups and an additional $\mathcal{O}(r^2 n)$ arithmetic operations.
Implementing an adaptive sampling rule for the pivots adds an additional $\mathcal{O}(n)$ entry look-ups and $\mathcal{O}(rn)$ operations, so it is effectively free.

A long line of results \cite{chen2025randomly,deshpande2007sampling,engler1997behavior,arthur2007kmeans,makarychev2020improved,gonzalez1985clustering} shows how adaptive sampling algorithms minimize the following distance functionals.
\begin{equation}
\label{eq:objectives}
\begin{aligned}
    \eta_{\rm RPC}(\mathsf{R}) 
    &= \sum_{i=1}^n d_{\bm{A}}\bigl(\bm{e}_i, \operatorname{span}\{\bm{e}_j\}_{j \in \mathsf{R}} \bigr)^2, \\
    \eta_{\rm SDS}(\mathsf{R}) 
    &= \sum_{i = 1}^n d_{\bm{A}}\bigl(\bm{e}_i, \{\bm{e}_j\}_{j \in \mathsf{R}}\bigr)^2,
     \\
    \eta_{\rm CPC}(\mathsf{R}) 
    &= \max_{1 \leq i \leq n} d_{\bm{A}}\bigl(\bm{e}_i, \operatorname{span}\{\bm{e}_j\}_{j \in \mathsf{R}} \bigr), \text{ or} \\
    \eta_{\rm FPS}(\mathsf{R}) 
    &= \max_{1 \leq i \leq n} d_{\bm{A}}\bigl(\bm{e}_i, \{\bm{e}_j\}_{j \in \mathsf{R}}\bigr).
\end{aligned}
\end{equation}
The approximation quality is rigorously bounded as follows.
\begin{theorem}[Adaptive sampling guarantees] \label{thm:worst_case}
    For any positive-semidefinite matrix $\bm{A} \in \mathbb{C}^{n \times n}$ and any cardinality $r \leq \operatorname{rank}(\bm{A})$, adaptive sampling generates a random or deterministic pivot set $\mathsf{R} \subseteq \{1, \ldots, n\}$ with cardinality $|\mathsf{R}| = r$ that satisfies
    \begin{align*}
        \frac{\mathbb{E}\bigl[ \eta_{\rm RPC}(\mathsf{R})\bigr]}{\min_{|\mathsf{S}| \leq r} \eta_{\rm RPC}(\mathsf{S})}
        &\leq 2^r,
        && \mathrm{[RPC]} \\
        \frac{\eta_{\rm CPC}(\mathsf{R})}{\min_{|\mathsf{S}| \leq r} \eta_{\rm CPC}(\mathsf{S})}
        &\leq r!,
        && \mathrm{[CPC]} \\
        \frac{\mathbb{E}\bigl[ \eta_{\rm SDS}(\mathsf{R})\bigr]}{\min_{|\mathsf{S}| \leq r} \eta_{\rm SDS}(\mathsf{S})}
        &\leq 5 (\log r + 2),
        && \mathrm{[SDS]} \\
        \frac{\eta_{\rm FPS}(\mathsf{R})}{\min_{|\mathsf{S}| \leq r} \eta_{\rm FPS}(\mathsf{S})}
        &\leq 2.
        && \mathrm{[FPS]}
    \end{align*}
    The distance functionals $\eta_{\rm RPC}$, $\eta_{\rm CPC}$, $\eta_{\rm SDS}$, and $\eta_{\rm FPS}$ are defined in \cref{eq:objectives}.
\end{theorem}
\begin{proof}
    The RPC error bound is \cite[Lem.~5.5]{chen2025randomly}.
    The CPC error bound is implied by \cite[Thm.~1]{deshpande2007sampling} in the limit $p \rightarrow \infty$; see also \cite[Thm.~2]{engler1997behavior}.
    The SDS error bound is \cite[Thm.~3.1]{arthur2007kmeans}, with an improved approximation factor due to \cite[Lem.~4.1]{makarychev2020improved}.
    The FPS error bound is \cite[Thm.~2.2]{gonzalez1985clustering}.
\end{proof}

\Cref{thm:worst_case} establishes rigorous approximation guarantees for the adaptive sampling algorithms.
The approximation factor depends only on the number of selected pivots $r$, independent of the dimensions $n$ of the target matrix.
As a weakness of this theory, however, the adaptive sampling algorithms optimize different functionals from the Kaporin condition number.
Empirically, adaptive sampling algorithms produce less accurate approximations than adaptive search (\cref{sec:empirical}).

The main reason to use adaptive sampling in practice is because of its speed; it is much faster than adaptive search.
In the future, we would ideally develop pivot choosers that are as fast as adaptive sampling and directly target the Kaporin condition number.

\subsection{Partial Cholesky + Vecchia} \label{sec:adding}

We can improve the partial Cholesky + diagonal approximation by incorporating a Vecchia residual component with a nonzero sparsity pattern $(\mathsf{Q}_i)_{i=1}^n$.
The Kaporin condition number for the partial Cholesky + Vecchia approximation, when finite, can be written as
\begin{equation}
\label{eq:kaporin_new}
    \kappa_{\rm Kap} = \prod_{d_{\tilde{\bm{R}}}(\bm{e}_i, \operatorname{span}\{\bm{e}_j\}_{j < i}) > 0}    \frac{d_{\tilde{\bm{R}}}\bigl(\bm{e}_i, \operatorname{span}\{\bm{e}_j\}_{j \in \mathsf{Q}_i} \bigr)^2}{d_{\tilde{\bm{R}}}\bigl(\bm{e}_i, \operatorname{span}\{\bm{e}_j\}_{j < i} \bigr)^2},
    \qquad \text{where } \tilde{\bm{R}} = \bm{P}^* \bm{R} \bm{P}.
\end{equation}
$\kappa_{\rm Kap}$ depends on a product of square weighted distances, where the weighting matrix is the pivoted residual matrix $\tilde{\bm{R}}$ from the partial Cholesky approximation.

To optimize the Kaporin condition number, we would ideally select each sparsity index set $\mathsf{Q}_i$ to make the distance
\begin{equation}
\label{eq:sq_dist}
    d_{\tilde{\bm{R}}}\bigl(\bm{e}_i, \operatorname{span}\{\bm{e}_j\}_{j \in \mathsf{Q}_i} \bigr)
\end{equation}
as small as possible.
Unfortunately, finding an index set with bounded cardinality $|\mathsf{Q}_i| \leq q$ that exactly minimizes the distance is NP-hard \cite[Thm.~1]{natarajan1995sparse}.
Nevertheless, 
we can consider approximation methods that might work well for matrices $\tilde{\bm{R}}$ arising in practice.

Two sparsity choosers appearing in the recent Vecchia literature are \emph{nearest neighbor search} \cite{guinness2018permutation} and \emph{orthogonal matching pursuit} \cite{huan2023sparse}.
Both approaches start with an empty sparsity index set $\mathsf{Q}_i = \emptyset$.
Then, in our context, NN search recursively adds a new index $j$ that solves the following distance minimization problem.
\begin{flalign}
\label{eq:nn}
    \qquad \qquad \mathsf{Q}_i \leftarrow \mathsf{Q}_i \cup \{j\},
    \quad \text{where } j \in \operatornamewithlimits{argmin}_{k \notin \mathsf{Q}_i} d_{\tilde{\bm{R}}}(\bm{e}_i, \bm{e}_k)
    && \text{[NN search]}.
\end{flalign}
In contrast, orthogonal matching pursuit uses a slightly different index selection rule.
\begin{flalign}
\label{eq:omp}
    \quad \mathsf{Q}_i \leftarrow \mathsf{Q}_i \cup \{j\},
    \quad \text{where } j \in \operatornamewithlimits{argmin}_{k \notin \mathsf{Q}_i} d_{\tilde{\bm{R}}}\bigl(\bm{e}_i, \operatorname{span}\{\bm{e}_j\}_{j \in \mathsf{Q}_i \cup \{k\}} \bigr)
    && \text{[OMP]}.
\end{flalign}
Comparing these two methods, OMP has the advantage that it directly targets the distance in the Kaporin condition number.
Yet, OMP has no a priori guarantees unless the off-diagonal entries of $\tilde{\bm{R}}$ decay quickly \cite[Thms. B and C]{tropp2004greed}.

Performing a full NN search or OMP algorithm is expensive because accessing the relevant entries of $\tilde{\bm{R}}$ is expensive.
To look up the entry $\tilde{\bm{R}}(i,j)$, we must first look up the $(i, j)$ entry of $\tilde{\bm{A}} = \bm{P}^*\bm{A} \bm{P}$ and then perform $\mathcal{O}(r)$ additional arithmetic operations using the rank-$r$ partial Cholesky approximation.
As such, the fastest direct implementation of the NN search requires $\mathcal{O}(n)$ entry look-ups of $\bm{A}$ and $\mathcal{O}(rn)$ arithmetic operations, independently of the number of nearest neighbors \cite{blum1973time}.
The fastest direct implementation of OMP requires $\mathcal{O}(n)$ entry look-ups of $\bm{A}$ and $\mathcal{O}((q^2 + q r) n)$ operations \cite{rubinstein2008batch}.

To reduce the computational burden, Huan and coauthors proposed restricting the Vecchia sparsity pattern to a ``candidate'' index set $\mathcal{C}_i \subseteq \{1, \ldots, i-1\}$, corresponding to the nearest neighbors in the $\tilde{\bm{A}}$-weighted distance \cite{huan2023sparse}.
Here we adapt their approach to define the following two-step optimization procedure.
\begin{enumerate}
    \item Build a restricted-cardinality candidate set $\mathsf{C}_i$ containing indices $j < i$ that minimize $d_{\tilde{\bm{A}}}(\bm{e}_i, \bm{e}_j)$.
    \item Apply the NN search \cref{eq:nn} or OMP algorithm \cref{eq:omp}, \emph{while restricting the search to the candidate indices $j \in \mathsf{C}_i$}.
\end{enumerate}
Step 1 requires $\mathcal{O}(n)$ entry look-ups and $\mathcal{O}(n)$ extra arithmetic operations, independently of the number of candidates.
Step 2 requires just $\mathcal{O}(n)$ arithmetic operations when the number of candidates is $|\mathsf{C}_i| = \mathcal{O}(n/r)$ for the NN search or $|\mathsf{C}_i| = \mathcal{O}(n/(q^2 + qr))$ for the OMP algorithm.
The two-step optimization approach is heuristic but is often effective in applications.

\section{Experiments}
\label{sec:empirical}

Kernel machine learning is a high-performing method for prediction and interpolation, but it is currently limited to small- or moderate-sized data sets due to computational cost \cite{scholkopf2001learning}.
Kernel matrices are typically dense and nearly singular, so they are challenging to handle with traditional iterative linear solvers or determinant estimation.
This section evaluates partial Cholesky + Vecchia as an accurate factored approximation that could potentially speed up kernel machine learning computations in the future.

\subsection{Experimental setup}

\begin{table}[t]
\centering
\caption{Data sets used in our experiments.
The machine learning data sets can be downloaded using a script developed for the paper \cite{diaz2024robust}, available at \href{https://github.com/eepperly/Robust-randomized-preconditioning-for-kernel-ridge-regression}{this link}.}
\label{tab:datasets}
\begin{tabular}{lll} \toprule
\textbf{Data set} & \textbf{Dimension} $d$ & \textbf{Source} \\
\midrule
\texttt{COMET\_MC\_SAMPLE} & 4 & LIBSVM \\ 
\texttt{cod\_rna} & 8 & LIBSVM \\
\texttt{Airlines\_DepDelay\_1M} & 9 & OpenML \\ 
\texttt{diamonds} & 9 & OpenML \\
\texttt{ACSIncome} & 11 & OpenML \\
\texttt{Click\_prediction\_small} & 11 & OpenML \\
\texttt{hls4ml\_lhc\_jets\_hlf} & 16 & OpenML \\
\texttt{Medical\_Appointment} & 18 & OpenML \\
\texttt{ijcnn1} & 22 & LIBSVM \\ 
\texttt{HIGGS} & 28 & LIBSVM \\ 
\texttt{creditcard} & 29 & OpenML \\ 
\texttt{sensorless} & 48 & LIBSVM \\
\texttt{MiniBoonNE} & 50 & OpenML \\
\texttt{covtype\_binary} & 54 & LIBSVM \\ 
\texttt{jannis} & 54 & OpenML \\ 
\texttt{YearPredictionMSD} & 90 & LIBSVM \\
\texttt{sensit\_vehicle} & 100 & LIBSVM \\
\texttt{yolanda} & 100 & OpenML \\
\texttt{connect\_4} & 126 & LIBSVM \\
\texttt{volkert} & 180 & OpenML \\ 
\texttt{santander} & 200 & OpenML \\
\texttt{MNIST} & 784 & OpenML \\
\bottomrule
\end{tabular}
\end{table}

To set up the kernel machine learning tests, we first downloaded 22 data sets from LIBSVM and OpenML.
For each data set, we normalized the predictors to have mean zero and variance one.
Then we subsampled the first $n = 20{,}000$ observations, yielding data points $\bm{z}_1, \ldots, \bm{z}_n \in \mathbb{R}^d$.
The number of predictors in each data set ranges from $d = 4$ to $d = 784$; see \cref{tab:datasets} for a full list.
For each regularization parameter $\mu \in \{10^{-10}, 10^{-6}, 10^{-3}\}$, we defined the strictly positive-definite kernel matrix $\bm{A} \in \mathbb{R}^{n \times n}$ with entries
\begin{equation*}
    \bm{A}(i,j) = \exp\biggl(-\frac{\lVert \bm{z}_i - \bm{z}_j \rVert^2}{2d}\biggr) + \mu\, \delta(i,j).
\end{equation*}
Then we generated several matrix approximations $\hat{\bm{A}}$, and we used these approximations as preconditioners for solving linear systems or calculating determinants.
See \cref{sec:optimality} for detailed descriptions of the preconditioning methods for linear systems and determinant calculations.
We ran three specific tests, described below.

\paragraph{1. PCG tests with label vectors}
We first solved linear systems $\bm{A} \bm{x} = \bm{b}$ where the response vectors $\bm{b} \in \mathbb{R}^n$ contained the labels provided with each data set.
We did not normalize $\bm{b}$ and even treated qualitative responses as quantitative responses.
These label vectors led to comparatively slow PCG convergence, so we terminated the method as soon as the relative error reached a tolerance $\lVert \bm{A} \hat{\bm{x}} - \bm{b} \rVert / \lVert \bm{b} \rVert \leq 10^{-3}$.
\paragraph{2. PCG tests with kernel vectors} Next we randomly sampled test data points  $\tilde{\bm{z}}_k \in \mathbb{R}^d$ for $k = 1, \ldots, 5$ and defined kernel response vectors $\bm{b}_k \in \mathbb{R}^n$ with entries
\begin{equation}
\label{eq:kernel_vectors}
\bm{b}_k(i) = \exp\biggl(-\frac{\lVert \bm{z}_i - \tilde{\bm{z}}_k \rVert^2}{2d}\biggr),
\qquad \text{for } i = 1, \ldots, n.
\end{equation}
These kernel vectors led to $10\times$ faster PCG convergence compared to the label vectors, so we ran the PCG method until reaching a tolerance $\lVert \bm{A} \hat{\bm{x}}_k - \bm{b}_k \rVert / \lVert \bm{b}_k \rVert \leq 10^{-4}$.
\paragraph{3. Log determinant tests}
Last we applied stochastic log determinant estimation to $\bm{A}$. 
We used ten randomly sampled test vectors $\bm{u}_1, \ldots, \bm{u}_{10} \in \mathbb{R}^n$, and we evaluated matrix logarithms using a Krylov–Ritz approximation \cite{frommer2016error} with depth $m = 100$; see \cref{sec:log} for a comparison with a larger depth $m = 1000$.
In our tests, we report the error in the normalized log determinant estimate $\frac{1}{n} \log (\hat{\bm{A}}) \approx \frac{1}{n} \log(\bm{A})$.

\subsection{Comparison of pivot choosers}

We first compared the performance of the five pivot choosers that were introduced in \cref{sec:partial_plus}.
\begin{itemize}
\item AS -- adaptive search;
\item RPC -- randomly pivoted Cholesky;
\item CPC -- column pivoted Cholesky;
\item SDS -- square distance sampling; and
\item FPS -- farthest point sampling.
\end{itemize}
We set the regularization parameter to $\mu = 10^{-3}$, and we applied the five pivot choosers with approximation rank $r = \lfloor n^{1/2}\rfloor = 141$ to help construct partial Cholesky + diagonal (\textbf{PC+V0}) and partial Cholesky + Vecchia (\textbf{PC+V1/4}) approximations.
For the Vecchia component, we selected
$q = \lfloor n^{1/4} \rfloor = 11$ off-diagonal nonzeros per row using the two-step OMP algorithm with $c = 10q$ candidates (\cref{sec:adding}).

\begin{figure}[t]
    \centering
    \includegraphics[width=\linewidth]{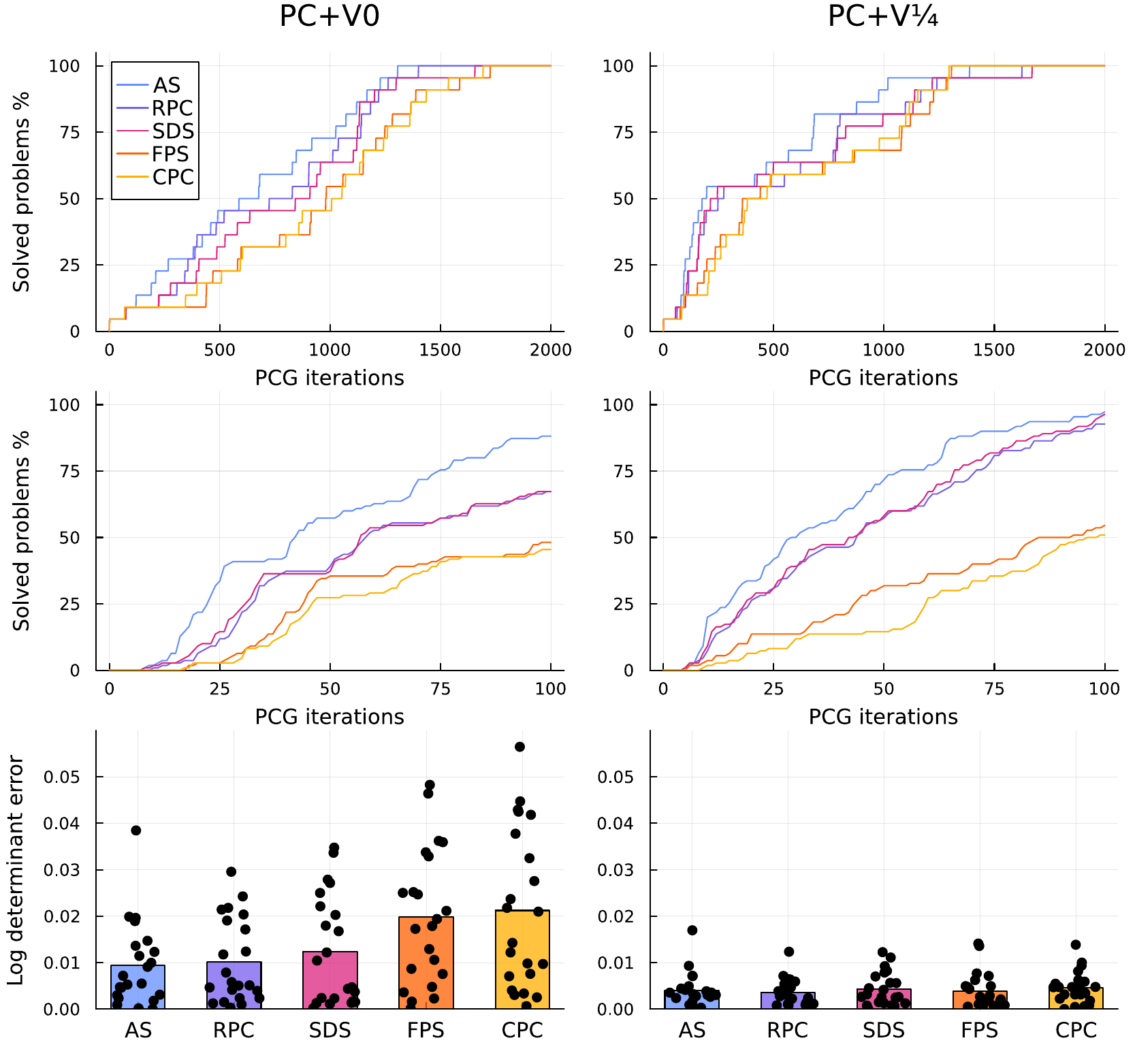}
    \caption{Comparison of five pivot choosers for building partial Cholesky + diagonal (\textbf{PC+V0}, left) and partial Cholesky + Vecchia (\textbf{PC+V1/4}, right) approximations.
    Top row: PCG tests with kernel vectors.
    Middle row: PCG tests with label vectors.
    Bottom row: log determinant tests.}
    \label{fig:PCG_pivots}
\end{figure}

\Cref{fig:PCG_pivots} shows that AS leads to the most accurate linear system solves and log determinant estimates of any available pivot chooser.
However, when using AS, the construction cost is two orders of magnitude higher than with any other method.
Therefore, AS is not a practical option for pivot selection, but this method's high accuracy suggests there might be an opportunity to improve the performance of practical pivot choosers in the future.

Among the practical pivot choosers, \cref{fig:PCG_pivots} shows that RPC and SDS lead to the most accurate linear system solves, and RPC leads to the most accurate determinant calculations.
FPS and CPC are less accurate across the tests, which is surprising since FPS was the main approach recently used to construct partial Cholesky + Vecchia approximations \cite{zhao2024adaptive,cai2025posterior}.
For example, RPC enables the solution to $1.4$--$1.7\times$ as many linear systems with kernel vectors after 100 iterations.

\subsection{Comparison of sparsity choosers}

We next compared the performance of the two sparsity choosers that were introduced in \cref{sec:adding}.
\begin{itemize}
\item NN -- nearest neighbor search with $c = 10\lfloor n^{1/2} \rfloor = 1{,}410$ candidates; and
\item OMP -- orthogonal matching pursuit with $c = 10 \lfloor n^{1/4} \rfloor = 110$ candidates.
\end{itemize}
We set the regularization parameter to $\mu = 10^{-3}$, and we applied the two sparsity choosers with $q = \lfloor n^{1/4} \rfloor = 11$ off-diagonal nonzeros per row to construct partial Cholesky + Vecchia (\textbf{PC+V1/4}) approximations.
For the partial Cholesky component, we used randomly pivoted Cholesky with an approximation rank $r = \lfloor n^{1/2} \rfloor = 141$.

\begin{figure}[t]
    \centering
    \includegraphics[width=\linewidth]{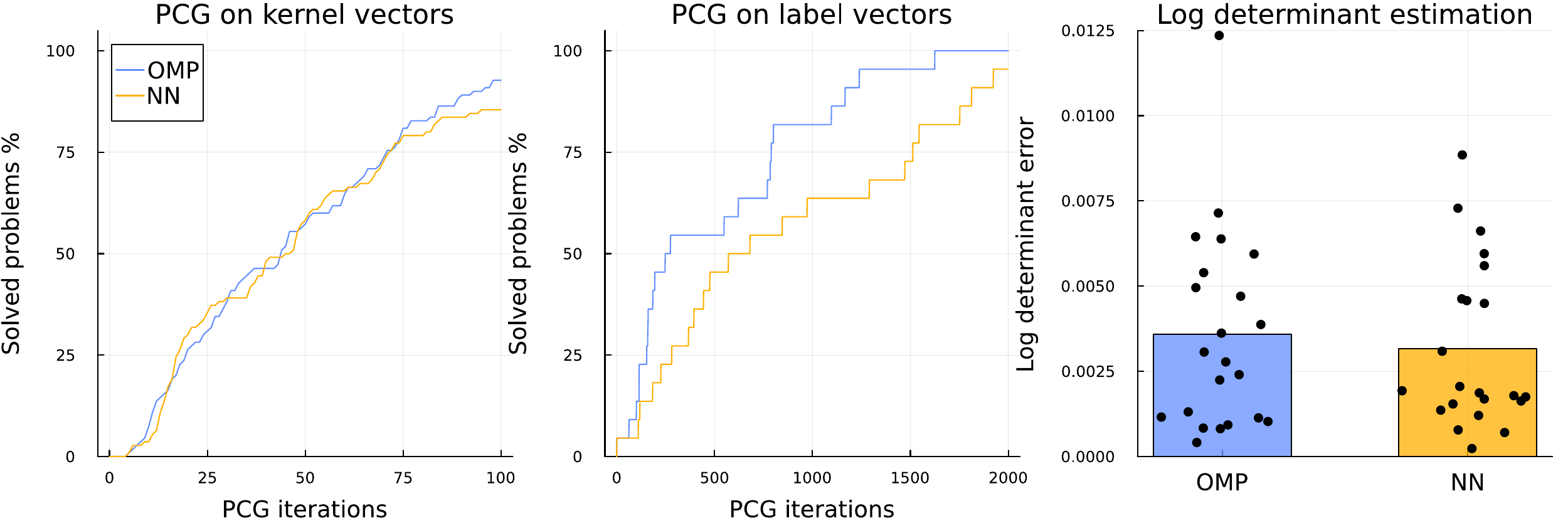}
    \caption{Comparison of two sparsity choosers for building a partial Cholesky + Vecchia (\textbf{PC+V1/4}) approximation.
    Left panel: PCG tests with label vectors.
    Middle panel: PCG tests with kernel vectors.
    Right panel: log determinant tests.}
    \label{fig:sparsity_choosers}
\end{figure}

\Cref{fig:sparsity_choosers} shows that OMP leads to the solution of $1.3\times$ as many linear systems compared to NN search, after 1000 iterations.
On the other tests, however, OMP and NN search perform similarly.

Going forward, we made the decision to use randomly pivoted Cholesky as a pivot chooser and orthogonal matching pursuit as a sparsity chooser, since this combination of methods generally yielded the highest accuracy.
As a theoretical benefit, OMP directly targets the minimization of the distances appearing inside the Kaporin condition number \cref{eq:kaporin_new}.
We found strong correspondence between the methods that directly minimize the Kaporin condition number and the methods that produce accurate linear solves and determinant estimates.

\subsection{Comparison to past preconditioners} \label{sec:vec_help}

There is a great variety of preconditioners for positive-semidefinite matrices \cite{tunnell2025empiricalstudyconjugategradient}, and a full empirical comparison is beyond the scope of the current paper.
Yet we can answer the more targeted question:
what is the best way to modify a partial Cholesky approximation to produce an effective preconditioner?

To answer this question, we compared three partial Cholesky-based preconditioners that we call \textbf{Frangella}, \textbf{D{\'i}az}, and \textbf{PC+V}.
\begin{itemize}
\item \textbf{Frangella}. Frangella et al. \cite[eq.~1.3]{frangella2023randomized} first proposed a partial Cholesky-based preconditioner for matrices of the form
\begin{equation*}
    \bm{A} = \bm{K} + \mu \mathbf{I}.
\end{equation*}
Here $\bm{K}$ is a positive-semidefinite kernel matrix and $\mu \geq 0$ is a regularization parameter.
In their work, Frangella and coauthors constructed a rank-$r$ partial Cholesky approximation $\hat{\bm{K}}$ \cite[Sec.~2.2.3]{frangella2023randomized} and then formed the preconditioner
\begin{equation}
\label{eq:bigger}
    \hat{\bm{A}} = \hat{\bm{K}} + \lambda_r(\hat{\bm{K}}) (\mathbf{I} - \bm{Q} \bm{Q}) + \mu \mathbf{I}.
\end{equation}
Here $\lambda_r(\hat{\bm{K}})$ is the $r$th largest eigenvalue of $\hat{\bm{K}}$, and $\bm{Q}$ is an orthonormal basis for $\hat{\bm{K}}$.

\item \textbf{D{\'i}az}. D{\'i}az et al. \cite[eq.~2.2]{diaz2024robust} later proposed a simplified partial Cholesky-based preconditioner
\begin{equation}
\label{eq:smaller}
    \hat{\bm{A}} = \hat{\bm{K}} + \mu \mathbf{I}.
\end{equation}
\item \textbf{PC+V}. Zhao et al. \cite{zhao2024adaptive} recently proposed the partial Cholesky + Vecchia preconditioner
\begin{equation*}
    \hat{\bm{A}} = \hat{\bm{A}}_{\rm part} + \hat{\bm{A}}_{\rm res}.
\end{equation*}
This method directly forms a rank-$r$ partial Cholesky approximation of $\bm{A} = \bm{K} + \mu \mathbf{I}$, without relying on a partial Cholesky approximation of $\bm{K}$.
\end{itemize}
We experimented with a range of regularization parameters $\mu \in \{10^{-10}, 10^{-6}, 10^{-3}\}$, and we constructed five different preconditioners based on randomly pivoted Cholesky with approximation rank $r = \lfloor n^{1/2} \rfloor = 141$.
In particular, we formed several partial Cholesky + Vecchia approximations with either $q = 0$ (\textbf{PC+V0}), $q = \lfloor n^{1/4} \rfloor = 11$ (\textbf{PC+V1/4}), or $q = \lfloor n^{1/3} \rfloor = 27$ (\textbf{PC+V1/3}) nonzero off-diagonal entries in the Vecchia component.
For the Vecchia construction, we used the two-step OMP algorithm with $c = 10q$ candidates.

\begin{figure}[t]
    \centering
    \includegraphics[width=\linewidth]{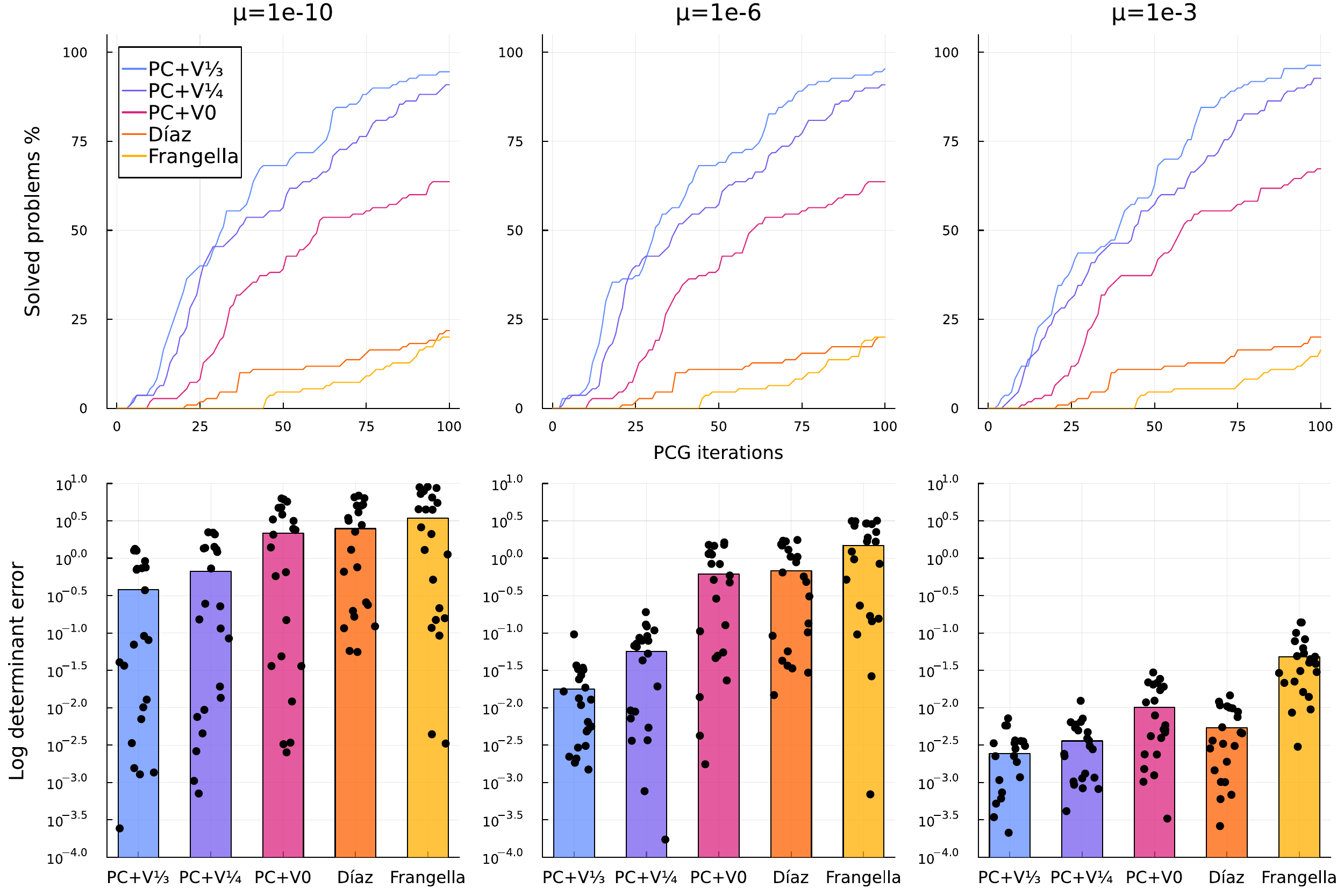}
    \caption{Comparison of five preconditioners
    based on randomly pivoted Cholesky with rank $r = \lfloor n^{1/2} \rfloor = 141$.
    Top row: PCG tests with kernel vectors.
    Bottom row: log determinant tests.
    Also see \cref{fig:intro} for the PCG tests with label vectors, using the same preconditioners.}
    \label{fig:outro}
\end{figure}

\Cref{fig:intro,fig:outro} show the results of comparing the five partial Cholesky-based preconditioners.
Across our tests,
\textbf{PC+V1/3} produced the highest accuracy linear solves and determinant estimates.
However, \textbf{PC+V1/3} has a construction cost of $\mathcal{O}(q^2rn) = \mathcal{O}(n^{13/6})$ arithmetic operations, so it is (slightly) more expensive than a linear-time algorithm.
Nevertheless, this approximation may prove useful in practice, and the high accuracy suggests that we can improve the partial Cholesky + Vecchia approximation through a creative strategy that further expands the sparsity pattern while controlling the computational expense.

\Cref{fig:intro,fig:outro} also show that the \textbf{PC+V0} and \textbf{PC+V1/4} preconditioners lead to consistently higher accuracy than the \textbf{Frangella} and \textbf{D{\'i}az} preconditioners while maintaining an asymptotic $\mathcal{O}(n^2)$ construction cost.
Compared to the \textbf{PC+V0} preconditioner, the \textbf{PC+V1/4} preconditioner improves the determinant accuracy by $3$--$11\times$ and leads to the solution to $1.4\times$ as many linear systems with kernel vectors after $100$ iterations.
Therefore, including even a small number of nonzero off-diagonal entries in the Vecchia component is empirically quite helpful.
This pattern of results is potentially surprising, since the number of off-diagonal nonzeros $q = \lfloor n^{1/4} \rfloor = 11$ is smaller than the number of predictors $d \in [4, 784]$ in $70\%$ of the machine learning data sets.

\section*{Acknowledgments}

The authors would like to thank Chris Cama\~no, Yifan Chen, Ethan N. Epperly, Christopher J. Geoga, and Florian Sch{\"a}fer for helpful discussions.

\bibliographystyle{siamplain}
\bibliography{references}

@article{zhao2024adaptive,
    author = {Zhao, Shifan and Xu, Tianshi and Huang, Hua and Chow, Edmond and Xi, Yuanzhe},
    title = {An Adaptive Factorized {N}ystr{\"o}m Preconditioner for Regularized Kernel Matrices},
    journal = {SIAM Journal on Scientific Computing},
    volume = {46},
    number = {4},
    pages = {A2351-A2376},
    year = {2024},
    doi = {10.1137/23M1565139},
}

@article{frangella2023randomized,
    author = {Frangella, Zachary and Tropp, Joel A. and Udell, Madeleine},
    title = {Randomized {N}ystr{\"o}m Preconditioning},
    journal = {SIAM Journal on Matrix Analysis and Applications},
    volume = {44},
    number = {2},
    pages = {718-752},
    year = {2023},
    doi = {10.1137/21M1466244},
}

@misc{diaz2024robust,
    title={Robust, randomized preconditioning for kernel ridge regression}, 
    author={Mateo Díaz and Ethan N. Epperly and Zachary Frangella and Joel A. Tropp and Robert J. Webber},
    year={2024},
    eprint={2304.12465},
    archivePrefix={arXiv},
    primaryClass={math.NA},
}

@article{ubaru2017fast,
    author = {Ubaru, Shashanka and Chen, Jie and Saad, Yousef},
    title = {Fast Estimation of \$tr(f(A))\$ via Stochastic Lanczos Quadrature},
    journal = {SIAM Journal on Matrix Analysis and Applications},
    volume = {38},
    number = {4},
    pages = {1075-1099},
    year = {2017},
    doi = {10.1137/16M1104974},
}

@inproceedings{arthur2007kmeans,
    author = {Arthur, David and Vassilvitskii, Sergei},
    title = {{k-means++}: {T}he advantages of careful seeding},
    year = {2007},
    booktitle = {Proceedings of the Eighteenth Annual ACM-SIAM Symposium on Discrete Algorithms},
    url = {https://dl.acm.org/doi/10.5555/1283383.1283494},
}

@inproceedings{higham1990analysis,
    author = {Nicholas J. Higham},
    title = {Analysis of the {Cholesky} Decomposition of a Semi-definite Matrix},
    booktitle = {Reliable Numerical Computation},
    editor = {M. G. Cox and S. J. Hammarling},
    pages = {161-185},
    publisher = {Oxford University Press},
    year = 1990,
    url = {https://eprints.maths.manchester.ac.uk/id/eprint/1193},
}

@article{gonzalez1985clustering,
    title = {Clustering to minimize the maximum intercluster distance},
    journal = {Theoretical Computer Science},
    volume = {38},
    pages = {293-306},
    year = {1985},
    issn = {0304-3975},
    doi = {10.1016/0304-3975(85)90224-5},
    author = {Teofilo F. Gonzalez},
}

@article{engler1997behavior,
    title = {The behavior of the {QR-}factorization algorithm with column pivoting},
    journal = {Applied Mathematics Letters},
    volume = {10},
    number = {6},
    pages = {7-11},
    year = {1997},
    doi = {10.1016/S0893-9659(97)00098-0},
    author = {H. Engler},
}

@inproceedings{makarychev2020improved,
    author = {Makarychev, Konstantin and Reddy, Aravind and Shan, Liren},
    title = {Improved guarantees for {k-means++} and {k-means++} parallel},
    year = {2020},
    booktitle = {Proceedings of the 34th International Conference on Neural Information Processing Systems},
    url = {https://dl.acm.org/doi/10.5555/3495724.3497078},
}

@inproceedings{deshpande2007sampling,
    author = {Deshpande, Amit and Varadarajan, Kasturi},
    title = {Sampling-based dimension reduction for subspace approximation},
    year = {2007},
    doi = {10.1145/1250790.1250884},
    booktitle = {Proceedings of the Thirty-Ninth Annual ACM Symposium on Theory of Computing},
}

@article{frommer2016error,
    author = {Frommer, Andreas and Schweitzer, Marcel},
    title = {Error bounds and estimates for {K}rylov subspace approximations of {S}tieltjes matrix functions},
    year = {2016},
    journal = {BIT Numerical Mathematics},
    pages = {865–892},
    volume = {56},
    number = {3},
    doi = {10.1007/s10543-015-0596-3},
}

@article{natarajan1995sparse,
    author = {Natarajan, B. K.},
    title = {Sparse Approximate Solutions to Linear Systems},
    journal = {SIAM Journal on Computing},
    volume = {24},
    number = {2},
    pages = {227-234},
    year = {1995},
    doi = {10.1137/S0097539792240406},
}

@article{chen2025randomly,
    author = {Chen, Yifan and Epperly, Ethan N. and Tropp, Joel A. and Webber, Robert J.},
    title = {Randomly pivoted {C}holesky: {P}ractical approximation of a kernel matrix with few entry evaluations},
    journal = {Communications on Pure and Applied Mathematics},
    volume = {78},
    number = {5},
    pages = {995-1041},
    doi = {10.1002/cpa.22234},
    year = {2025}
}

@article{huan2023sparse,
    author = {Huan, Stephen and Guinness, Joseph and Katzfuss, Matthias and Owhadi, Houman and Sch\"{a}ofer, Florian},
    title = {Sparse Inverse {C}holesky Factorization of Dense Kernel Matrices by Greedy Conditional Selection},
    journal = {SIAM/ASA Journal on Uncertainty Quantification},
    volume = {13},
    number = {3},
    pages = {1649-1679},
    year = {2025},
    doi = {10.1137/23M1606253},
}

@article{schafer2021sparse,
    author = {Sch\"{a}fer, Florian and Katzfuss, Matthias and Owhadi, Houman},
    title = {Sparse {C}holesky Factorization by {K}ullback--{L}eibler Minimization},
    journal = {SIAM Journal on Scientific Computing},
    volume = {43},
    number = {3},
    pages = {A2019-A2046},
    year = {2021},
    doi = {10.1137/20M1336254},
}

@techreport{rubinstein2008batch,
    author      = {Rubinstein, Ron and Zibulevsky, Michael and Elad, Michael},
    title       = {Efficient Implementation of the {K-SVD} Algorithm using Batch Orthogonal Matching Pursuit},
    institution = {Technion -- Israel Institute of Technology},
    year        = {2008},
    number      = {CS-2008-08},
    url = {https://csaws.cs.technion.ac.il/~ronrubin/Publications/KSVD-OMP-v2.pdf},
}

@article{yeremin2000factorized,
  title={Factorized sparse approximate inverse preconditionings. {III.} {I}terative construction of preconditioners},
  author={Yeremin, Alex Yu and Kolotilina, Lily Yu and Nikishin, AA},
  journal={Journal of Mathematical Sciences},
  volume={101},
  pages={3237--3254},
  year={2000},
  doi={10.1007/BF02672769},
}

@article{vecchia1988estimation,
    title={Estimation and model identification for continuous spatial processes},
    author = {Vecchia, A. V.},
    journal={Journal of the Royal Statistical Society Series B: Statistical Methodology},
    volume={50},
    number={2},
    pages={297--312},
    year={1988},
    doi = {10.1111/j.2517-6161.1988.tb01729.x},
}

@article{dempster1972,
    doi = {https://doi.org/10.2307/2528966},
    author = {A. P. Dempster},
    journal = {Biometrics},
    number = {1},
    pages = {157--175},
    title = {Covariance Selection},
    volume = {28},
    year = {1972}
}

@article{axelsson2000sublinear,
title={On the sublinear and superlinear rate of convergence of conjugate gradient methods},
volume={25}, 
doi={10.1023/a:1016694031362},
number={1–4}, 
journal={Numerical Algorithms}, 
author={Axelsson, Owe and Kaporin, Igor},
year={2000},
pages={1–22},
}

@article{cai2025posterior,
author = {Cai, Difeng and Chow, Edmond and Xi, Yuanzhe},
title = {Posterior Covariance Structures in Gaussian Processes},
journal = {SIAM Journal on Matrix Analysis and Applications},
volume = {46},
number = {2},
pages = {1640-1673},
year = {2025},
doi = {10.1137/24M1684918},
}

@article{kaporin1994new,
author = {Kaporin, I. E.},
title = {New convergence results and preconditioning strategies for the conjugate gradient method},
journal = {Numerical Linear Algebra with Applications},
volume = {1},
number = {2},
pages = {179-210},
doi = {10.1002/nla.1680010208},
year = {1994}
}

@article{golub1965numerical,
    title={Numerical methods for solving linear least squares problems},
    volume={7}, 
    doi={10.1007/bf01436075},
    number={3},
    journal={Numerische Mathematik}, 
    author={Golub, G.},
    year={1965}, 
    pages={206–216},
}

@article{sun2016statistically,
    author = {Ying Sun and Michael L. Stein},
    title = {Statistically and Computationally Efficient Estimating Equations for Large Spatial Datasets},
    journal = {Journal of Computational and Graphical Statistics},
    volume = {25},
    number = {1},
    pages = {187--208},
    year = {2016},
    doi = {10.1080/10618600.2014.975230},
}

@article{tropp2004greed,
    author={Tropp, J.A.},
    journal={IEEE Transactions on Information Theory}, 
    title={Greed is good: {A}lgorithmic results for sparse approximation}, 
    year={2004},
    volume={50},
    number={10},
    pages={2231-2242},
    doi={10.1109/TIT.2004.834793},
}

@book{scholkopf2001learning,
    author = {Schölkopf, Bernhard and Smola, Alexander J.},
    title = {Learning with Kernels: Support Vector Machines, Regularization, Optimization, and Beyond},
    publisher = {The MIT Press},
    year = {2001},
    doi = {10.7551/mitpress/4175.001.0001},
}

@book{golub2013matrix,
author = {Golub, Gene H. and Van Loan, Charles F.},
title = {Matrix Computations - 4th Edition},
publisher = {Johns Hopkins University Press},
year = {2013},
doi = {10.1137/1.9781421407944},
address = {Philadelphia, PA},
edition = {4},
}

@article{guinness2018permutation,
  title={Permutation and grouping methods for sharpening {G}aussian process approximations},
  author={Guinness, Joseph},
  journal={Technometrics},
  volume={60},
  number={4},
  pages={415--429},
  year={2018},
  doi={10.1080/00401706.2018.1437476},
}

@misc{tunnell2025empiricalstudyconjugategradient,
      title={An Empirical Study of Conjugate Gradient Preconditioners for Solving Symmetric Positive Definite Systems of Linear Equations}, 
      author={Marc A. Tunnell and David F. Gleich},
      year={2025},
      eprint={2505.20696},
      archivePrefix={arXiv},
      primaryClass={math.NA},
      url={https://arxiv.org/abs/2505.20696}, 
}

@article{blum1973time,
    title = {Time bounds for selection},
    journal = {Journal of Computer and System Sciences},
    volume = {7},
    number = {4},
    pages = {448-461},
    year = {1973},
    doi = {10.1016/S0022-0000(73)80033-9},
    author = {Manuel Blum and Robert W. Floyd and Vaughan Pratt and Ronald L. Rivest and Robert E. Tarjan},
}

\appendix

\section{Proof of Kaporin optimality theory} \label{app:vecchia_proof}
This section establishes the optimality of the Vecchia approximation as presented as \cref{thm:optimality}.
For notational simplicity, this section assumes the permutation is $\bm{P} = \mathbf{I}$.
Otherwise, we can permute the indices of $\bm{A}$ and apply the proof to the permuted matrix.

To begin, we need formulas for the pseudoinverse and volume of a matrix in terms of its inverse Cholesky decomposition.

\begin{lemma}[Pseudoinverse and volume formulas]
\label{lem:formulas}
Fix a positive-semidefinite matrix with inverse Cholesky decomposition $\bm{A} = \bm{C}^{-1} \bm{D} \bm{C}^{-*}$.
Let $\bm{Q}$ be an orthonormal basis for $\operatorname{range}(\bm{A})$ and set
$\bm{R} = \mathbf{I}(\cdot, \mathsf{G})$, where $\mathsf{G} \subseteq \{1, \ldots, n\}$ picks out the nonzero elements of $\bm{D}$.
Then
\begin{equation}
\label{eq:vol_formula}
    \bm{A}^+ = \bm{Q} \bm{Q}^* \bm{C}^* \bm{D}^+ \bm{C} \bm{Q} \bm{Q}^* \quad \text{and} 
    \quad \operatorname{vol}(\bm{A})
    = \frac{\operatorname{vol}(\bm{D})}{\det(\bm{R}^* \bm{Q})^2}.
\end{equation}
\end{lemma}
\begin{proof}
    The pseudoinverse $\bm{A}^+$ is the unique positive-semidefinite matrix that satisfies $\bm{A}^+ \bm{x} = \bm{0}$ if $\bm{x} \perp \operatorname{range}(\bm{Q})$
    and $\bm{A}\bm{A}^+ \bm{x} = \bm{x}$ if $\bm{x} \in \operatorname{range}(\bm{Q})$.
    The matrix $\bm{Q} \bm{Q}^* \bm{C}^* \bm{D}^+ \bm{C} \bm{Q} \bm{Q}^*$ satisfies these conditions, confirming \cref{eq:vol_formula}.
    
    Next observe that $\bm{R}^* \bm{C}^{-1} \bm{R}$ is a lower triangular matrix with ones on the diagonal; hence, it has determinant one.
    Consequently,
    \begin{equation*}
        1 = \det \bigl(\bm{R}^* \bm{C}^{-1} \bm{R}\bigr) = \det \bigl(\bm{R}^* \bm{Q} \bm{Q}^* \bm{C}^{-1} \bm{R} \bigr) 
        = \det(\bm{R}^* \bm{Q}) \det(\bm{Q}^* \bm{C}^{-1} \bm{R}).
    \end{equation*}
    The volume of $\bm{A}$ is given by
    \begin{align*}
        \operatorname{vol}(\bm{A}) 
        &= \operatorname{vol}\bigl(\bm{Q} \bm{Q}^* \bm{C}^{-1} \bm{R} \bm{R}^* \bm{D} \bm{R} \bm{R}^* \bm{C}^{-*} \bm{Q} \bm{Q}^*\bigr) \\
        &= \det\bigl(\bm{Q}^* \bm{C}^{-1} \bm{R}\bigr)^2 \det\bigl(\bm{R}^* \bm{D} \bm{R}\bigr)
        = \operatorname{vol}(\bm{D}) / \det(\bm{R}^* \bm{Q})^2.
    \end{align*}
    This confirms \cref{eq:vol_formula} and completes the proof.
\end{proof}

Next we provide explicit formulas for the trace and volume of $\bm{A} \hat{\bm{A}}^+$, when $\hat{\bm{A}}$ is an approximation with the same range as $\bm{A}$.

\begin{lemma}[Trace and volume formulas] \label{lem:trace_volume}
    Consider two inverse Cholesky decompositions with the same range, $\bm{A} = \bm{C}^{-1} \bm{D} \bm{C}^{-*}$ and $\hat{\bm{A}} = \hat{\bm{C}}^{-1} \hat{\bm{D}} \hat{\bm{C}}^{-*}$.
    Then
    \begin{equation*}
    \operatorname{vol}\bigl(\hat{\bm{A}}^+ \bm{A}\bigr)
    = \frac{\operatorname{vol}\bigl(\bm{A}\bigr)}{\operatorname{vol}\bigl(\hat{\bm{A}}\bigr)}
    =\frac{\operatorname{vol}\bigl(\bm{D}\bigr)}{\operatorname{vol}\bigl(\hat{\bm{D}}\bigr)} \quad \text{and} \quad
    \operatorname{tr}\bigl(\hat{\bm{A}}^+ \bm{A}\bigr)
    = \operatorname{tr}\bigl(\hat{\bm{D}}^+ \hat{\bm{C}} \bm{A}\hat{\bm{C}}^* \bigr).
    \end{equation*}
\end{lemma}
\begin{proof}
First we argue that $\operatorname{range}(\hat{\bm{D}}) = \operatorname{range}(\bm{D})$.
To that end, introduce the affine subspace
\begin{equation*}
    \mathcal{A}_i = \{\bm{v} \in \mathbb{C}^n\,|\, \bm{v}(i) = 1, \bm{v}(j)= 0 \text{ for } j < i\}.
\end{equation*}
for $i = 1, \ldots, n$.
Because of the lower triangular factorization, the range of $\bm{A}$ contains a vector $\bm{v} \in \mathsf{A}_i$ if and only if $\bm{D}(i,i) > 0$.
Similarly, the range of $\hat{\bm{A}}$ contains a vector $\bm{v} \in \mathsf{A}_i$ if and only if $\hat{\bm{D}}(i,i) > 0$.
Since $\operatorname{range}(\bm{A}) = \operatorname{range}(\hat{\bm{A}})$, it follows that $\operatorname{range}(\hat{\bm{D}}) = \operatorname{range}(\bm{D})$.

Next we take an orthonormal basis $\bm{Q}$ for $\operatorname{range}\bigl(\hat{\bm{A}}\bigr) = \operatorname{range}(\bm{A})$ and an orthonormal basis $\bm{R}$ for $\operatorname{range}\bigl(\hat{\bm{D}}\bigr) = \operatorname{range}(\bm{D})$. 
To calculate $\operatorname{vol}\bigl(\hat{\bm{A}}^+ \bm{A}\bigr)$, we observe
\begin{align*}
    \operatorname{vol}\bigl(\hat{\bm{A}}^+ \bm{A}\bigr)
    &= \operatorname{vol}\bigl(\bm{Q} \bm{Q}^* \hat{\bm{A}}^+ \bm{Q} \bm{Q}^* \bm{A} \bm{Q} \bm{Q}^* \bigr) \\
    &= \operatorname{det}\bigl(\bm{Q}^* \hat{\bm{A}}^+ \bm{Q}\bigr)
   \operatorname{det}\bigl(\bm{Q}^* \bm{A} \bm{Q}\bigr)
    = \operatorname{vol}\bigl(\hat{\bm{A}}^+\bigr) \operatorname{vol}\bigl(\bm{A}\bigr) 
    = \frac{\operatorname{vol}\bigl(\bm{A}\bigr)}{\operatorname{vol}\bigl(\hat{\bm{A}}\bigr)}.
\end{align*}
The last line uses the fact that $\operatorname{vol}\bigl(\hat{\bm{A}}^+\bigr)$ is the inverse of  $\operatorname{vol}\bigl(\hat{\bm{A}}\bigr)$.
Next we apply the exact formulas for $\operatorname{vol}\bigl(\bm{A}\bigr)$ and $\operatorname{vol}\bigl(\hat{\bm{A}}\bigr)$ given in \cref{lem:formulas}.
\begin{equation*}
    \frac{\operatorname{vol}\bigl(\bm{A}\bigr)}{\operatorname{vol}\bigl(\hat{\bm{A}}\bigr)} 
    = \frac{\operatorname{vol}\bigl(\bm{D}\bigr) / \operatorname{det}(\bm{R}^* \bm{Q})^2}{\operatorname{vol}\bigl(\hat{\bm{D}}\bigr)  / \operatorname{det}(\bm{R}^* \bm{Q})^2}
    = \frac{\operatorname{vol}\bigl(\bm{D}\bigr)}{\operatorname{vol}\bigl(\hat{\bm{D}}\bigr)}.
\end{equation*}
Similarly, we calculate $\operatorname{tr}(\hat{\bm{A}}^+ \bm{A})$ by applying the exact formula for the pseudoinverse $\hat{\bm{A}}$ given in \cref{lem:formulas}.
\begin{align*}
    & \operatorname{tr}\bigl(\hat{\bm{A}}^+ \bm{A}\bigr)
    = \operatorname{tr}\bigl(\bm{Q} \bm{Q}^* \hat{\bm{C}}^* \hat{\bm{D}}^+ \hat{\bm{C}} \bm{Q} \bm{Q}^* \bm{A} \bigr) \\
    &= \operatorname{tr}\bigl(\hat{\bm{D}}^+ \hat{\bm{C}} \bm{Q} \bm{Q}^* \bm{A} \bm{Q} \bm{Q}^*  \hat{\bm{C}}^* \bigr)
    = \operatorname{tr}\bigl(\hat{\bm{D}}^+ \hat{\bm{C}} \bm{A}\hat{\bm{C}}^* \bigr)
\end{align*}
This completes the proof.
\end{proof}

Now we establish the main result, \cref{thm:optimality}.

\begin{proof}[Proof of \cref{thm:optimality}]

Let $\bm{A} = \bm{C}^{-1} \bm{D} \bm{C}^{-*}$ be an \emph{exact} inverse Cholesky decomposition.
We will derive the inverse Cholesky approximation $\hat{\bm{A}} = \hat{\bm{C}}^{-1} \hat{\bm{D}} \hat{\bm{C}}^{-*}$ with the given sparsity pattern $(\mathsf{S}_i)_{i=1}^n$ that minimizes $\kappa_{\rm Kap}$.
We will focus on the case where $\bm{A}$ and $\hat{\bm{A}}$ share the same range, because otherwise $\kappa_{\rm Kap} = \infty$.

We partition the indices $\{1, \ldots, n\}$ into a set of ``good'' indices $\mathsf{G} = \{i\,|\,\bm{D}(i,i) > 0\}$ and a set of ``bad'' indices $\mathsf{B} = \{i\,|\,\bm{D}(i,i) = 0\}$.
Since $\operatorname{rank}(\bm{A}) = \operatorname{rank}(\bm{D})$, the number of good indices is $r = \operatorname{rank}(\bm{A})$ and the number of bad indices is $n - r$.
We will consider the good and bad indices separately.

We first consider what happens to a bad index $i \in \mathsf{B}$.
Since $\bm{D}(i,i) = 0$ and $\bm{A}$ and $\hat{\bm{A}}$ share a nullspace, it follows that
\begin{equation*}
    \bm{0} = \bm{C}(i, \cdot) \bm{A} = \bm{C}(i, \cdot) \hat{\bm{A}} = \bm{C}(i, \cdot) \hat{\bm{C}}^{-1} \hat{\bm{D}}.
\end{equation*}
We are using the fact that $\hat{\bm{A}}$ has the same left nullspace as $\hat{\bm{C}}^{-1} \hat{\bm{D}}$.
Next we observe that $\bm{C} \hat{\bm{C}}^{-1}$ is lower triangular with ones on the diagonal, so we must have $(\bm{C} \hat{\bm{C}}^{-1})(i,i) = 1$ and consequently $\hat{\bm{D}}(i,i) = 0$.
Since $\hat{\bm{D}}(i,i) = 0$ and also $\bm{A}$ and $\hat{\bm{A}}$ share a nullspace, it follows that
\begin{equation*}
    \bm{0} = \hat{\bm{C}}(i, \cdot) \hat{\bm{A}} = \hat{\bm{C}}(i, \cdot) \bm{A},
\end{equation*}
and we can write
\begin{equation*}
    \begin{bmatrix} \hat{\bm{C}}(i, \mathsf{S}_i) & 1 \end{bmatrix} \begin{bmatrix} \bm{A}(\mathsf{S}_i, \mathsf{S}_i) & \bm{A}(\mathsf{S}_i, i) \\
    \bm{A}(i, \mathsf{S}_i) & \bm{A}(i, i) \end{bmatrix}
    = \begin{bmatrix} \bm{0} & \hat{\bm{D}}(i,i) \end{bmatrix}.
\end{equation*}
This is precisely the definition of the Vecchia approximation given in \cref{def:vecchia_approximation}.

Now we focus on the good indices $i \in \mathsf{G}$.
To that end, we rewrite $\kappa_{\rm Kap}$ using the volume and trace formulas from \cref{lem:trace_volume}.
\begin{equation*}
    \kappa_{\rm Kap} = \frac{\bigl(\frac{1}{r} \operatorname{tr} \bigl(\bm{A} \hat{\bm{A}}^+\bigr)\bigr)^r}
    {\operatorname{vol} \bigl(\bm{A} \hat{\bm{A}}^+\bigr)}
    = \frac{\operatorname{vol}\bigl(\hat{\bm{D}}\bigr)\operatorname{tr}\bigl(\hat{\bm{D}}^+ \hat{\bm{C}}^* \bm{A}\hat{\bm{C}} \bigr)^r}{r^r \operatorname{vol}\bigl(\bm{D}\bigr)}.
\end{equation*}
We calculate the logarithmic derivative
\begin{equation*}
    \partial_{\hat{\bm{D}}(i,i)} \log(\kappa_{\rm Kap}) = \frac{1}{\hat{\bm{D}}(i,i)} - \frac{r\bigl(\hat{\bm{C}} \bm{A} \hat{\bm{C}}^*\bigr)(i,i)}{\hat{\bm{D}}(i,i)^2\operatorname{tr}\bigl(\hat{\bm{D}}^+ \hat{\bm{C}} \bm{A} \hat{\bm{C}}^* \bigr)}.
\end{equation*}
The logarithmic derivative is negative for small $\hat{\bm{D}}(i,i)$ and positive for large $\hat{\bm{D}}(i,i)$, and it achieves a zero value when 
\begin{equation*}
    \frac{\bigl(\hat{\bm{C}} \bm{A} \hat{\bm{C}}^*\bigr)(i,i)}{\hat{\bm{D}}(i,i)}
    = \frac{\operatorname{tr}\bigl(\hat{\bm{D}}^+ \hat{\bm{C}} \bm{A}\hat{\bm{C}}^* \bigr)}{r}
\end{equation*}
We select the minimizer $\hat{\bm{D}}(i,i) = \bigl(\hat{\bm{C}} \bm{A} \hat{\bm{C}}^*\bigr)(i,i)$ for each $i \in \mathsf{G}$.
Then the Kaporin condition number becomes
\begin{equation*}
    \kappa_{\rm Kap} = \frac{\operatorname{vol}\bigl(\hat{\bm{D}}\bigr)}{\operatorname{vol}\bigl(\bm{D}\bigr)}
    = \frac{1}{\operatorname{vol}\bigl(\bm{D}\bigr)} \prod_{i \in \mathsf{G}} \bigl(\hat{\bm{C}} \bm{A} \hat{\bm{C}}^*\bigr)(i,i).
\end{equation*}
Last, we expand the square.
\begin{align*}
    & \bigl(\hat{\bm{C}} \bm{A} \hat{\bm{C}}^*\bigr)(i,i) \\
    &= \hat{\bm{C}}(i, \mathsf{S}_i) \bm{A}(\mathsf{S}_i, \mathsf{S}_i)
    \hat{\bm{C}}(i, \mathsf{S}_i)^*
    + \hat{\bm{C}}(i, \mathsf{S}_i) \bm{A}(\mathsf{S}_i, i) 
    + \bm{A}(i, \mathsf{S}_i)
    \hat{\bm{C}}(i, \mathsf{S}_i)^*
    + \bm{A}(i,i) \\
    &= \bigl[\hat{\bm{C}}(i, \mathsf{S}_i) \bm{A}(\mathsf{S}_i, \mathsf{S}_i) + \bm{A}(i, \mathsf{S}_i)\bigr] \bm{A}(\mathsf{S}_i, \mathsf{S}_i)^+ \bigl[\bm{A}(\mathsf{S}_i, \mathsf{S}_i) \hat{\bm{C}}(\mathsf{S}_i, i) + \bm{A}(\mathsf{S}_i, i)\bigr] \\
    & \qquad + \bigl[\bm{A}(i, i) - \bm{A}(i, \mathsf{S}_i) \bm{A}(\mathsf{S}_i, \mathsf{S}_i)^+ \bm{A}(\mathsf{S}_i, i)\bigr].
\end{align*}
Here we have used the fact that $\bm{A}(\mathsf{S}_i, i) \in \operatorname{range}\bigl(\bm{A}(\mathsf{S}_i, \mathsf{S}_i)\bigr)$, because $\bm{A}$ is positive-semidefinite.
The minimizer of $\bigl(\hat{\bm{C}} \bm{A} \hat{\bm{C}}^*\bigr)(i,i)$ is a row vector $\hat{\bm{C}}(i, \mathsf{S}_i)$ that is characterized by
\begin{equation*}
    \hat{\bm{C}}(i, \mathsf{S}_i) \bm{A}(\mathsf{S}_i, \mathsf{S}_i) + \bm{A}(i, \mathsf{S}_i) = \bm{0}.
\end{equation*}
The minimum value is
\begin{equation}
\label{eq:minimum}
    d_{\bm{A}}\bigl(\bm{e}_i, \operatorname{span}\{\bm{e}_j\}_{j \in \mathsf{S}_i} \bigr)^2
    = \bigl(\hat{\bm{C}} \bm{A} \hat{\bm{C}}^*\bigr)(i,i)
    = \hat{\bm{C}}(i, \mathsf{S}_i)
    \bm{A}(\mathsf{S}_i, i)
    + \bm{A}(i,i).
\end{equation}
At this point, we have shown that
\begin{equation*}
    \begin{bmatrix} \hat{\bm{C}}(i, \mathsf{S}_i) & 1 \end{bmatrix} \begin{bmatrix} \bm{A}(\mathsf{S}_i, \mathsf{S}_i) & \bm{A}(\mathsf{S}_i, i) \\
    \bm{A}(i, \mathsf{S}_i) & \bm{A}(i, i) \end{bmatrix}
    = \begin{bmatrix} \bm{0} & \hat{\bm{D}}(i,i) \end{bmatrix}.
\end{equation*}
This is precisely the definition of the Vecchia approximation given in \cref{def:vecchia_approximation}.

In conclusion, the Vecchia approximation achieves the minimal Kaporin condition number, which can be written as
\begin{equation*}
    \kappa_{\rm Kap} = \frac{1}{\operatorname{vol}\bigl(\bm{D}\bigr)} \prod_{i \in \mathsf{G}} d_{\bm{A}}\bigl(\bm{e}_i, \operatorname{span}\{\bm{e}_j\}_{j \in \mathsf{S}_i} \bigr)^2,
\end{equation*}
using \cref{eq:minimum}.
In the particular case where the sparsity pattern is $\mathsf{S}_i = \{1, \ldots, i-1\}$ for $i = 1, \ldots, n$, the minimal Kaporin condition number is $\kappa_{\rm Kap} = 1$.
Hence,
\begin{equation*}
    \operatorname{vol}\bigl(\bm{D}\bigr)
    = \prod_{i \in \mathsf{G}} d_{\bm{A}}\bigl(\bm{e}_i, \operatorname{span}\{\bm{e}_j\}_{j < i} \bigr)^2.
\end{equation*}
This completes the proof.
\end{proof}

\section{Applications of the Kaporin condition number}

The section proves several upper bounds for linear algebra calculations in terms of the Kaporin condition number.

\subsection{Proof of \texorpdfstring{\cref{prop:direct}}{Proposition 3.6}}
\label{sec:direct}

We can use the fact that $\bm{b} - \bm{A} \bm{x}_\star$ is orthogonal to $\operatorname{range}(\bm{A}) = \operatorname{range}(\hat{\bm{A}})$ to calculate
\begin{align*}
    \lVert \hat{\bm{x}} - \bm{x}_{\star} \rVert_{\bm{A}}
    &= \lVert \bm{x}_0 + \hat{\bm{A}}^+[\bm{b} - \bm{A} \bm{x}_0] - \bm{x}_{\star} \rVert_{\bm{A}} \\
    &= \bigl\lVert \bigl[\mathbf{I} - \hat{\bm{A}}^+ \bm{A} \bigr] \bigl[\bm{x}_0 - \bm{x}_{\star}\bigr] \bigr\rVert_{\bm{A}} \\
    &= \bigl\lVert \bigl[\mathbf{I} - \bm{A}^{1/2} \hat{\bm{A}}^+ \bm{A}^{1/2} \bigr] \bm{A}^{1/2} \bigl[\bm{x}_0 - \bm{x}_{\star}\bigr] \bigr\rVert \\
    &\leq \lVert \mathbf{I} - \bm{A}^{1/2} \hat{\bm{A}}^+ \bm{A}^{1/2}\rVert \lVert \hat{\bm{x}} - \bm{x}_0 \rVert_{\bm{A}}
\end{align*}
The eigenvalues of $\bm{A}^{1/2} \hat{\bm{A}}^+ \bm{A}^{1/2}$ are the eigenvalues of $\bm{A} \hat{\bm{A}}^+$, which we write as $\lambda_1 \geq \cdots \geq \lambda_r > 0$.
Therefore,
\begin{equation*}
    \frac{\lVert \hat{\bm{x}} - \bm{x}_{\star} \rVert_{\bm{A}}}{\lVert \hat{\bm{x}} - \bm{x}_0 \rVert_{\bm{A}}}
    \leq \lVert \mathbf{I} - \bm{A}^{1/2} \hat{\bm{A}}^+ \bm{A}^{1/2}\rVert = \max_{1 \leq i \leq r} |1 - \lambda_i|,
\end{equation*}
which gives a sharp upper bound.

Given that $\sum_{j=1}^r \lambda_j = r$, the concavity of $x \mapsto \log x$ shows that $\sum_{j=1}^r \log(\lambda_j)$ is maximized when the eigenvalues $\lambda_j$ for $j \neq i$ are equal.
Again using $\sum_{j=1}^r \lambda_j = r$, we obtain
\begin{equation*}
    -\log(\kappa_{\rm Kap}) = \sum_{j=1}^r \log(\lambda_j)
    \leq \log(\lambda_i) + (r-1)\log\biggl(\frac{r-\lambda_i}{r-1}\biggr).
\end{equation*}
Over the interval $(0, r]$, we can bound $\log(x)$ from above by a concave quadratic that passes through $(1, 0)$ and $(r, \frac{r}{2} - \frac{1}{2r})$, so
\begin{equation*}
    \log(\lambda_i) \leq -1 + \lambda_i - \frac{1}{2r}(1 - \lambda_i)^2.
\end{equation*}
Also, $\log(1 + x) \leq x$ holds globally, so
\begin{equation*}
    (r-1) \log\biggl(\frac{r-\lambda_i}{r-1}\biggr)
    = (r-1) \log\biggl(1 + \frac{1-\lambda_i}{r-1}\biggr) \leq 1-\lambda_i.
\end{equation*}
It follows that
\begin{equation*}
    -\log(\kappa_{\rm Kap})
    \leq \log(\lambda_i) + (r-1)\log\biggl(\frac{r-\lambda_i}{r-1}\biggr) \leq
    -\frac{1}{2r}(1 - \lambda_i)^2.
\end{equation*}
We conclude that $(1 - \lambda_i)^2 \leq 2r \log(\kappa_{\rm Kap})$ for each $i = 1, \ldots, r$, which completes the proof.

\subsection{Proof of \texorpdfstring{\cref{prop:pcg}}{Theorem 3.7}} \label{sec:pcg_proof}

The starting point is a classic error bound for PCG iterates \cite[eq.~(3.4)]{axelsson2000sublinear}.
\begin{equation*}
    \frac{\lVert \bm{x}_t - \bm{x}_{\star} \rVert_{\bm{A}}}{\lVert \bm{x}_0 - \bm{x}_{\star} \rVert_{\bm{A}}} \leq \max_{1 \leq i \leq r} \bigl|p_t(\lambda_i)|.
\end{equation*}
Here, $p_t$ is any degree-$t$ polynomial satisfying $p_t(0) = 1$, and $\lambda_1 \geq \cdots \geq \lambda_r$ are the sorted positive eigenvalues of $\bm{A} \hat{\bm{A}}^+$.
In the case of even $t$, we construct
\begin{equation*}
    p_t(\lambda) = \prod_{i=1}^{t/2} \biggl(1 - \frac{\lambda}{\lambda_i}\biggr)\biggl(1 - \frac{\lambda}{\lambda_{r+1-i}}\biggr).
\end{equation*}
In the case of odd $t$, we set $t_\star = (t+1)/2$ and construct
\begin{equation*}
    p_t(\lambda) = \biggl(1 - \frac{2\lambda}{\lambda_{t_{\star}} + \lambda_{r+1-t_{\star}}} \biggr)\prod_{i=1}^{(t-1)/2} \biggl(1 - \frac{\lambda}{\lambda_i}\biggr)\biggl(1 - \frac{\lambda}{\lambda_{r+1-i}}\biggr).
\end{equation*}
Now observe that
\begin{equation*}
    \max_{1 \leq i \leq r} \bigl|p_t(\lambda_i)|
    = \max_{\lfloor t/2 \rfloor \leq i \leq r + 1 - \lfloor t/2 \rfloor} \bigl|p_t(\lambda_i)|,
\end{equation*}
and the right-hand side is bounded by the product of terms
\begin{equation*}
    \max_{\lambda_i \leq \lambda \leq \lambda_{r+1 - i}} \biggl(\frac{\lambda}{\lambda_i} - 1\biggr)\biggl(1 - \frac{\lambda}{\lambda_{r+1-i}}\biggr) = \frac{(\lambda_i - \lambda_{r+1-i})^2}{4 \lambda_i \lambda_{r+1-i}}
\end{equation*}
and potentially a term
\begin{equation*}
    \max_{\lambda_{t_{\star}} \leq \lambda \leq  \lambda_{r+1-t_{\star}}}\biggl|1 - \frac{2\lambda}{\lambda_{t_{\star}} + \lambda_{r+1-t_{\star}}} \biggr|
    = \frac{\lambda_{t_{\star}} - \lambda_{r+1-t_{\star}}}{\lambda_{t_{\star}} + \lambda_{r+1-t_{\star}}}
    \leq \frac{\lambda_{t_{\star}} - \lambda_{r+1-t_{\star}}}{2 (\lambda_{t_{\star}} \lambda_{r+1-t_{\star}})^{1/2}},
\end{equation*}
where the last line follows from the arithmetic-geometric mean inequality
\begin{equation*}
    (\lambda_{t_{\star}} \lambda_{r+1-t_{\star}})^{1/2} \leq \frac{\lambda_{t_{\star}} + \lambda_{r+1-t_{\star}}}{2}.
\end{equation*}
In the case of even $t$, we can use the arithmetic-geometric mean inequality twice to write
\begin{align*}
    & \prod_{i=1}^{t/2} \biggl(\frac{4 \lambda_i \lambda_{r+1-i}}{(\lambda_i + \lambda_{r+1-i})^2}\biggr)^{2/t}
    + \prod_{i=1}^{t/2} \biggl(\frac{(\lambda_i - \lambda_{r+1-i})^2}{(\lambda_i + \lambda_{r+1-i})^2}\biggr)^{2/t} \\
    &\leq \frac{2}{t} \sum_{i=1}^{t/2} \frac{4 \lambda_i \lambda_{r+1-i} + (\lambda_i - \lambda_{r+1-i})^2}{(\lambda_i + \lambda_{r+1-i})^2} = 1,
\end{align*}
and consequently
\begin{equation*}
    1 + \biggl[\frac{\lVert \bm{x}_t - \bm{x}_{\star} \rVert_{\bm{A}}}{\lVert \bm{x}_0 - \bm{x}_{\star} \rVert_{\bm{A}}}\biggr]^{2/t} \leq 1 + \prod_{i=1}^{t/2} \biggl(\frac{(\lambda_i - \lambda_{r+1-i})^2}{4 \lambda_i \lambda_{r+1-i}}\biggr)^{2/t}
    \leq \prod_{i=1}^{t/2} \biggl(\frac{(\lambda_i + \lambda_{r+1-i})^2}{4 \lambda_i \lambda_{r+1-i}}\biggr)^{2/t}.
\end{equation*}
In the case of odd $t$, we can use the generalized arithmetic-geometric mean inequality with weights $1/t, 2/t, \ldots, 2/t$ twice to write
\begin{align*}
    & \biggl(\frac{4 \lambda_{t_{\star}} \lambda_{r+1-t_{\star}}}{(\lambda_{t_{\star}} + \lambda_{r+1-t_{\star}})^2}\biggr)^{1/t} \prod_{i=1}^{(t-1)/2} \biggl(\frac{4 \lambda_i \lambda_{r+1-i}}{(\lambda_i + \lambda_{r+1-i})^2}\biggr)^{2/t} \\
    & \quad + \biggl(\frac{(\lambda_{t_{\star}} -  \lambda_{r+1-t_{\star}})^2}{(\lambda_{t_{\star}} + \lambda_{r+1-t_{\star}})^2}\biggr)^{1/t}\prod_{i=1}^{(t-1)/2} \biggl(\frac{(\lambda_i - \lambda_{r+1-i})^2}{(\lambda_i + \lambda_{r+1-i})^2}\biggr)^{2/t} \\
    &\leq \frac{1}{t} \frac{4 \lambda_{t_{\star}} \lambda_{r+1-t_{\star}} + (\lambda_{t_{\star}} - \lambda_{r+1-t_{\star}})^2}{(\lambda_{t_{\star}} + \lambda_{r+1-t_{\star}})^2}
    + \frac{2}{t} \sum_{i=1}^{(t-1)/2} \frac{4 \lambda_i \lambda_{r+1-i} + (\lambda_i - \lambda_{r+1-i})^2}{(\lambda_i + \lambda_{r+1-i})^2} = 1,
\end{align*}
and consequently
\begin{align*}
    & 1 + \biggl[\frac{\lVert \bm{x}_t - \bm{x}_{\star} \rVert_{\bm{A}}}{\lVert \bm{x}_0 - \bm{x}_{\star} \rVert_{\bm{A}}}\biggr]^{2/t} \\
    &\leq 1 + \biggl(\frac{(\lambda_{t_{\star}} -  \lambda_{r+1-t_{\star}})^2}{4 \lambda_{t_{\star}} \lambda_{r+1-t_{\star}}}\biggr)^{1/t} \prod_{i=1}^{(t-1)/2} \biggl(\frac{(\lambda_i - \lambda_{r+1-i})^2}{4 \lambda_i \lambda_{r+1-i}}\biggr)^{2/t} \\
    &\leq \biggl(\frac{(\lambda_{t_{\star}} +  \lambda_{r+1-t_{\star}})^2}{4 \lambda_{t_{\star}} \lambda_{r+1-t_{\star}}}\biggr)^{1/t} \prod_{i=1}^{t/2} \biggl(\frac{(\lambda_i + \lambda_{r+1-i})^2}{4 \lambda_i \lambda_{r+1-i}}\biggr)^{2/t}.
\end{align*}
Last, when $r$ is even, we can use the arithmetic-geometric mean inequality two more times to write
\begin{equation*}
    \biggl(1 + \biggl[\frac{\lVert \bm{x}_t - \bm{x}_{\star} \rVert_{\bm{A}}}{\lVert \bm{x}_0 - \bm{x}_{\star} \rVert_{\bm{A}}}\biggr]^{2/t}\biggr)^{t/2} \leq \frac{1}{\prod_{i=1}^r \lambda_i} \prod_{i=1}^{r/2} \biggl(\frac{\lambda_i + \lambda_{r+1-i}}{2}\biggr)^2
    \leq \kappa_{\rm Kap}.
\end{equation*}
Similarly when $r$ is odd, we can use the generalized arithmetic-geometric mean inequality two more times to obtain
\begin{equation*}
    \biggl(1 + \biggl[\frac{\lVert \bm{x}_t - \bm{x}_{\star} \rVert_{\bm{A}}}{\lVert \bm{x}_0 - \bm{x}_{\star} \rVert_{\bm{A}}}\biggr]^{2/t}\biggr)^{t/2} \leq \frac{1}{\prod_{i=1}^r \lambda_i} \lambda_{(r+1)/2} \prod_{i=1}^{(r-1)/2} \biggl(\frac{\lambda_i + \lambda_{r+1-i}}{2}\biggr)^2
    \leq \kappa_{\rm Kap}.
\end{equation*}
We conclude by recalling that
\begin{align*}
    2x &= \int_0^x 2 \mathrm{d}y \leq \int_0^x \frac{2 \mathrm{d}y}{1 - y^2}
    = \int_0^x \biggl[\frac{1}{1 - y} + \frac{1}{1 + y}\biggr] \mathrm{d}y \\
    &= \biggl[-\log(1-y) + \log(1 + y)\biggr]^{y=x}_{y=0} = \log\biggl(\frac{1+x}{1-x}\biggr)
\end{align*}
and therefore ${\rm e}^{2x} - 1 \leq \frac{2x}{1-x}$. Therefore, we make the calculation
\begin{align*}
    & \frac{\lVert \bm{x}_t - \bm{x}_{\star} \rVert_{\bm{A}}}{\lVert \bm{x}_0 - \bm{x}_{\star} \rVert_{\bm{A}}} \leq \bigl(\kappa_{\rm Kap}^{2/t} - 1\bigr)^{t/2}
    \leq \biggl(\frac{2 \log (\kappa_{\rm Kap})/t}{1 - \log (\kappa_{\rm Kap})/t}\biggr)^{t/2} \\
    &= \biggl(\frac{2 \log (\kappa_{\rm Kap})}{2t/3 + [t/3 - \log (\kappa_{\rm Kap})]}\biggr)^{t/2} \leq \biggl(\frac{3 \log (\kappa_{\rm Kap})}{t}\biggr)^{t/2}
\end{align*}
on the event that $t/3 \geq \log(\kappa_{\rm Kap})$.
On the other hand, if $t/3 < \log(\kappa_{\rm Kap})$, the result \cref{eq:superlinear} holds vacuously.
This completes the proof.

\subsection{Proof of \texorpdfstring{\cref{prop:stochastic_det}}{Theorem 3.8}}
\label{sec:det_proof}
We prove the result for complex-valued vectors; the real-valued case is similar.
Consider the matrix $\bm{B} = \hat{\bm{A}}^{-1/2} \bm{A} \hat{\bm{A}}^{-1/2}$ which has eigendecomposition
\begin{equation*}
\bm{B} = \bm{Q} \bm{\Lambda} \bm{Q}^*
\quad \text{where} \quad 
\bm{\Lambda} = \operatorname{diag}(\lambda_1, \ldots, \lambda_n).
\end{equation*}
We need to prove the mean square error bound
\begin{equation*}
    \mathbb{E}\Biggl|\operatorname{tr}(\log \bm{B} )
    - \frac{1}{t}\sum_{i=1}^t \bm{z}_i^* (\log \bm{B}) \bm{z}_i\Biggr|^2
    \leq \frac{4 \log(\kappa_{\rm Kap})}{t}.
\end{equation*}
To that end, we introduce a complex Gaussian vector $\bm{\omega} \sim \mathcal{N}(\bm{0}, \mathbf{I})$ and set $\bm{z} = \sqrt{n}\, \bm{\omega} / \lVert \bm{\omega}\rVert$.
It suffices to show 
\begin{equation}
\label{eq:mean_variance}
    \mathbb{E}[\bm{z}^* (\log \bm{B}) \bm{z}] = \operatorname{tr}(\log \bm{B})
    \quad \text{and} \quad
    \operatorname{Var}[\bm{z}^* (\log \bm{B}) \bm{z}] \leq 4 \log(\kappa_{\rm Kap})
\end{equation}
when the Kaporin condition number is $\log(\kappa_{\rm Kap}) \leq n$.

We first observe that rotational invariance implies $\bm{z} \stackrel{\mathcal{D}}{=} \bm{Q}^* \bm{z}$ so 
\begin{equation*}
    \bm{z}^* (\log\bm{B}) \bm{z} \stackrel{\mathcal{D}}{=} \bm{z}^* (\log\bm{\Lambda}) \bm{z}
    = \sum_{i=1}^n \log (\lambda_i) |\bm{z}(i)|^2.
\end{equation*}
Then we check the stochastic trace estimator is unbiased.
\begin{equation}
\label{eq:mean}
    \mathbb{E}\Biggl[\sum_{i=1}^n \log (\lambda_i) |\bm{z}(i)|^2 \Biggr]
    = \sum_{i=1}^n \log (\lambda_i)\, \mathbb{E}|\bm{z}(i)|^2 = \sum_{i=1}^n \log (\lambda_i) = \operatorname{tr}(\log \bm{B}).
\end{equation}
This confirms the first part of \cref{eq:mean}.

To bound the variance, we make a calculation using the Gaussian vector $\bm{\omega} \sim \mathcal{N}(\bm{0}, \mathbf{I})$.
\begin{align*}
    &\mathbb{E}\Biggl|\sum_{i=1}^n \log (\lambda_i) |\bm{\omega}(i)|^2 \Biggr|^2
    - \Biggl|\sum_{i=1}^n \log (\lambda_i) \Biggr|^2
    = \operatorname{Var}\Biggl[\sum_{i=1}^n \log (\lambda_i) |\bm{\omega}(i)|^2 \Biggr] \\
    &= \sum_{i=1}^n (\log \lambda_i)^2 \operatorname{Var}\bigl[|\bm{\omega}(i)|^2 \bigr] = \sum_{i=1}^n (\log \lambda_i)^2.
\end{align*}
Here we have used the independence of $\bm{\omega}(i)$ variables and the identity $\operatorname{Var}\bigl[|\bm{\omega}(i)|^2 \bigr] = 1$, which holds for complex Gaussians.
Since $\bm{\omega}$ has independent length and direction, we calculate
\begin{equation*}
    \mathbb{E}\Biggl|\sum_{i=1}^n \log (\lambda_i) |\bm{\omega}(i)|^2 \Biggr|^2 = \frac{\mathbb{E} \lVert \bm{\omega} \rVert^4}{n^2} \,\mathbb{E}\Biggl|\sum_{i=1}^n \log (\lambda_i) |\bm{z}(i)|^2 \Biggr|^2
    = \frac{n^2 + n}{n^2} \,\mathbb{E}\Biggl|\sum_{i=1}^n \log (\lambda_i) |\bm{z}(i)|^2 \Biggr|^2,
\end{equation*}
where again we have used the fact that $\operatorname{Var}\bigl[|\bm{\omega}(i)|^2 \bigr] = 1$.
By rearrangement, it follows
\begin{equation*}
    \mathbb{E}\Biggl|\sum_{i=1}^n \log(\lambda_i) |\bm{z}(i)|^2 \Biggr|^2 = \frac{n}{n+1}\Biggl[\sum_{i=1}^n (\log \lambda_i)^2 + \Biggl(\sum_{i=1}^n \log(\lambda_i) \Biggr)^2 \Biggr].
\end{equation*}
Subtracting the square mean \cref{eq:mean} shows that
\begin{align*}
    & \operatorname{Var}\Biggl[\sum_{i=1}^n \log(\lambda_i) |\bm{z}(i)|^2\Biggr] \\
    &= \frac{n}{n+1}\Biggl[ \sum_{i=1}^n (\log \lambda_i)^2 - \frac{1}{n} \Biggl(\sum_{i=1}^n \log(\lambda_i) \Biggr)^2  \Biggr]
    = \frac{n}{n+1} \sum_{i=1}^n (\log \lambda_i')^2,
\end{align*}
where we have introduced $\lambda_i' = \lambda_i / \bigl(\prod_{j=1}^n \lambda_j\bigr)^{1/n}$ and we observe that $\prod_{i=1}^n \lambda_i' = 1$ and $\frac{1}{n} \sum_{i=1}^n \lambda_i' = \kappa_{\rm Kap}^{1/n}$.
Then recall the standard identity that
\begin{equation*}
    1 + t + \frac{1}{2} t^2 \leq {\rm e}^t
    \quad \text{and therefore} \quad
    (\log u)^2 \leq 2(u - 1 - \log u).
\end{equation*}
We obtain the variance bound
\begin{align*}
    \operatorname{Var}\Biggl[\sum_{i=1}^n \log(\lambda_i) |\bm{z}(i)|^2\Biggr]
    \leq \frac{2n}{n+1} \sum_{i=1}^n \bigl(\lambda_i' - 1 - \log \lambda_i' \bigr)
    = \frac{2n^2}{n+1} \bigl(\kappa_{\rm Kap}^{1/n} - 1\bigr).
\end{align*}
Last observe that $t \mapsto ({\rm e}^t - 1)/t$ is strictly increasing for $t > 0$ and therefore
\begin{equation*}
    \frac{\kappa_{\rm Kap}^{1/n}-1}{\log (\kappa_{\rm Kap})/n}
    \leq 2
    \quad \text{if} \quad \log(\kappa_{\rm Kap}) \leq 1.42 n.
\end{equation*}
This completes the proof.

\section{Comparison of Krylov depths} \label{sec:log}
\begin{figure}[t]
    \centering
    \includegraphics[width=\linewidth]{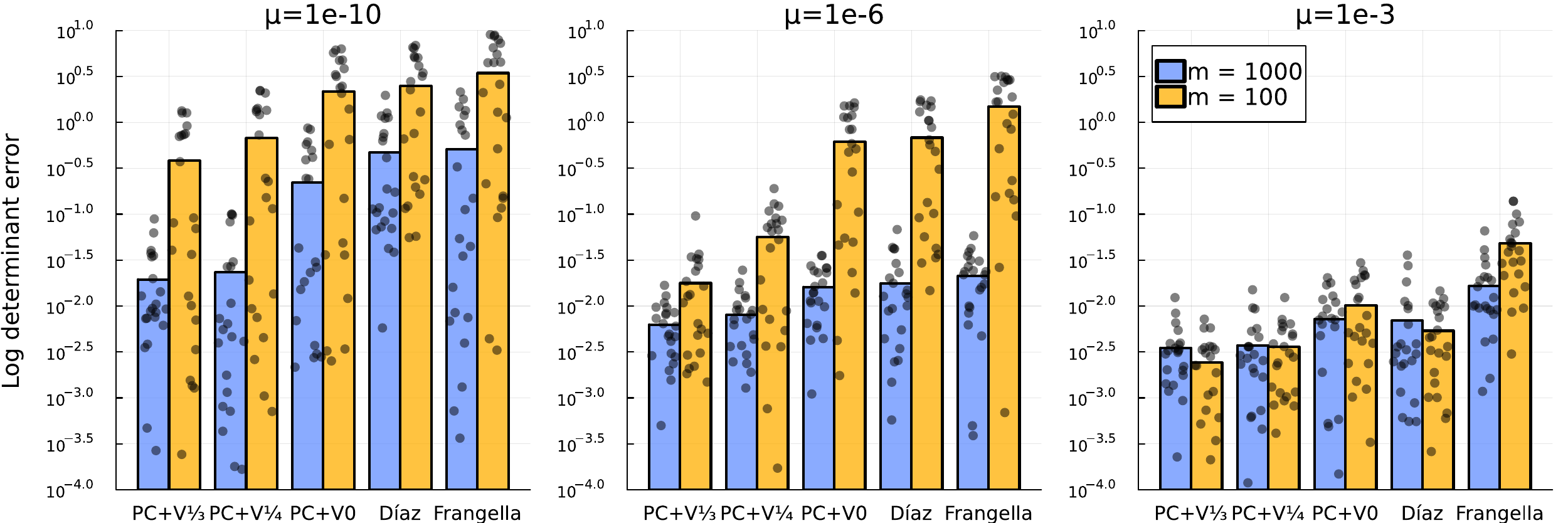}
    \caption{Comparison of log determinant calculations with Krylov depths $m = 1000$ (blue) and $m = 100$ (yellow).}
    \label{fig:krylov}
\end{figure}
\Cref{fig:krylov} compares stochastic determinant estimates obtained with two different Krylov depth parameters, $m = 1000$ and $m = 100$, while keeping the number of samples fixed at $10$.
Increasing the Krylov depth reduces the bias of the estimator, since it causes the Krylov approximation of the matrix log to become more accurate.
See \cref{sec:determinant,sec:vec_help} for the details of these calculations.

The $10\times$ larger Krylov subspace sometimes improves the accuracy of the log determinant by $1$--$2$ orders of magnitude, but the effect depends on $\mu$ and the preconditioner.
The accuracy is nearly the same for $m = 1000$ and $m = 100$ when $\mu = 10^{-3}$ or when $\mu = 10^{-6}$ and we use the highly accurate \textbf{PC+V1/3} preconditioner.
While insightful for comparison, we anticipate that the Krylov depth $m = 1000$ is too expensive to be practical in most situations, so we emphasize the depth $m = 100$ results in \cref{sec:empirical}.

\end{document}